\documentclass[reqno, 11pt, a4paper]{amsart}

\usepackage{latexsym,ifthen,xspace}
\usepackage{amsmath,amssymb,amsthm}
\usepackage{enumerate, calc}
\usepackage[colorlinks=true, pdfstartview=FitV, linkcolor=blue,
  citecolor=blue, urlcolor=blue]{hyperref}
\usepackage[text={33pc,605pt},centering]{geometry}    %11pt
\usepackage{graphicx}

\newtheorem{theorem}{Theorem}[section]
\newtheorem{lemma}[theorem]{Lemma}
\newtheorem{prop}[theorem]{Proposition}
\newtheorem{assumption}[theorem]{Assumption}
\newtheorem{corro}[theorem]{Corollary}

\theoremstyle{definition}

\newtheorem{example}[theorem]{Example}

\theoremstyle{remark}
\newtheorem{remark}[theorem]{Remark}

\numberwithin{equation}{section}

\DeclareMathAlphabet{\mathsl}{OT1}{cmss}{m}{sl}
\SetMathAlphabet{\mathsl}{bold}{OT1}{cmss}{bx}{sl}

\newcommand{\al}{\ensuremath{\alpha}}

\newcommand{\ga}{\ensuremath{\gamma}}
\newcommand{\de}{\ensuremath{\delta}}

\newcommand{\ze}{\ensuremath{\zeta}}

\renewcommand{\th}{\ensuremath{\theta}}

\newcommand{\ka}{\ensuremath{\kappa}}
\newcommand{\la}{\ensuremath{\lambda}}

\newcommand{\si}{\ensuremath{\sigma}}

\newcommand{\om}{\ensuremath{\omega}}

\newcommand{\ve}{\ensuremath{\varepsilon}}

\newcommand{\Ga}{\ensuremath{\Gamma}}

\newcommand{\Om}{\ensuremath{\Omega}}
\newcommand{\cB}{\ensuremath{\mathcal B}}

\newcommand{\cE}{\ensuremath{\mathcal E}}
\newcommand{\cF}{\ensuremath{\mathcal F}}

\newcommand{\cL}{\ensuremath{\mathcal L}}

\newcommand{\bbN}{\ensuremath{\mathbb N}}

\newcommand{\bbR}{\ensuremath{\mathbb R}}

\newcommand{\bbZ}{\ensuremath{\mathbb Z}} 
\newcommand{\me}{\ensuremath{\mathrm{e}}}

\newcommand{\md}{\ensuremath{\mathrm{d}}}

\newcommand{\scpr}[3]{%
  \ensuremath{%
    \big\langle
      #1, #2
    \big\rangle_{\raisebox{0.1ex}{$\scriptstyle \ell^{\raisebox{.1ex}{$\scriptscriptstyle 2$}} (#3)$}}
  }
}
\newcommand{\norm}[3]{%
  \ensuremath{%
    \big\lVert
      #1
    \big\rVert_{\raisebox{-.0ex}{$\scriptstyle \ell^{\raisebox{.2ex}{$\scriptscriptstyle #2$}} (#3)$}}
  }
}
\newcommand{\Norm}[2]{%
  \ensuremath{%
    \big\lVert
      #1
    \big\rVert_{\raisebox{-.0ex}{$\scriptstyle #2$}}
  }
}

\DeclareMathOperator{\mean}{\mathbb{E}}

\DeclareMathOperator{\prob}{\mathbb{P}}
\DeclareMathOperator{\Prob}{\mathrm{P}}

\DeclareMathOperator{\supp}{\mathrm{supp}}

\DeclareMathOperator{\argmax}{\mathrm{arg max}}

\newcommand{\av}[1]{\mathop{\mathrm{av}}(#1)}

\newcommand{\ldef}{\ensuremath{\mathrel{\mathop:}=}}
\newcommand{\rdef}{\ensuremath{=\mathrel{\mathop:}}}

\newcommand{\indicator}{%
  \ensuremath{%
    \mathchoice{1\mskip-4mu\mathrm l}
    {1\mskip-4mu\mathrm l}
    {1\mskip-4.5mu\mathrm l}
    {1\mskip-5mu\mathrm l}
  }
}

\begin{document}

\title{Heat kernel estimates for random walks with degenerate weights}

%    Remove any unused author tags.

%   author one information
\author{Sebastian Andres}
\address{Rheinische Friedrich-Wilhelms Universit\"at Bonn}
\curraddr{Endenicher Allee 60, 53115 Bonn}
\email{andres@iam.uni-bonn.de}
\thanks{}

%    author two information
\author{Jean-Dominique Deuschel}
\address{Technische Universit\"at Berlin}
\curraddr{Strasse des 17. Juni 136, 10623 Berlin}
\email{deuschel@math.tu-berlin.de}
\thanks{}

%    author three information
\author{Martin Slowik}
\address{Technische Universit\"at Berlin}
\curraddr{Strasse des 17. Juni 136, 10623 Berlin}
\email{slowik@math.tu-berlin.de}
\thanks{}

\subjclass[2010]{39A12,  60J35, 60K37, 82C41}

\keywords{random walk, heat kernel, Moser iteration}

\date{\today}

\dedicatory{}

\begin{abstract}
  We establish Gaussian-type upper bounds on the heat kernel for a continuous-time random walk on a graph with unbounded weights under an integrability assumption. For the proof we use Davies' perturbation method, where we show a maximal inequality for the perturbed heat kernel via Moser iteration.
\end{abstract}

\maketitle

\tableofcontents

\section{Introduction}

A well known theorem by Delmotte \cite{De99} states that Gaussian bounds on the heat kernel hold for random walks on locally finite graphs, provided the jump rates are uniformly elliptic, that is the transition probabilities are uniformly bounded and bounded away from zero.  In a recent work \cite{Fo11}, Folz showed Gaussian upper bounds for the heat kernel of continuous-time, elliptic random walks with arbitrary speed measure under the assumption  that on-diagonal upper bounds for the heat kernel at two points are given and the speed measure is uniformly bounded from below.  In the present paper we relax the uniform ellipticity condition and show a Gaussian-type upper bound for constant-speed and variable-speed random walks with unbounded jump rates satisfying a certain integrability condition. 

\subsection{Setting and Result}
Let $G = (V, E)$ be an infinite, connected, locally finite graph with vertex set $V$ and (non-oriented) edge set $E$.  We will write $x \sim y$ if $\{x,y\} \in E$.  The graph $G$ is endowed with the counting measure, i.e.\ the measure of $A \subset V$ is simply the number $|A|$ of elements in $A$.  Further, we denote by $B(x,r)$ the closed ball with center $x$ and radius $r$ with respect to the natural graph distance $d$, that is $B(x,r) \ldef \{y \in V \mid d(x,y) \leq r\}$.

For a given set $B \subset V$, we define the \emph{relative} internal boundary of $A \subset B$ by
\begin{align*}
  \partial_B A
  \;\ldef\;
  \big\{
    x \in A \;\big|\;
    \exists\, y \in B \setminus A\; \text{ s.th. }\; \{x,y\} \in E
  \big\}
\end{align*}
and we simply write $\partial A$ instead of $\partial_V A$.  Throughout the paper we will make the following assumption on $G$.
\begin{assumption}\label{ass:graph}
  The graph $G$ satisfies the following conditions:
  \begin{itemize}
  \item[(i)] volume regularity of order $d$ for large balls, that is there exists $d \geq 2$ and $C_{\mathrm{reg}} \in (0, \infty)$ such that for all $x \in V$ there exists $N_1(x) < \infty$ with
    \begin{align}\label{eq:ass:vd}
      C^{-1}_{\mathrm{reg}}\, n^d  \;\leq\; |B(x,n)| \;\leq\; C_{\mathrm{reg}}\, n^d
      \qquad \forall\, n \geq N_1(x). 
    \end{align}
  \item[(ii)] local Sobolev inequality $(S_{d'}^1)$ for large balls, that is there exists $d' \geq d$ and $C_{\mathrm{S_1}} \in (0, \infty)$ such that for all $x \in V$ the following holds. There exists $N_2(x) < \infty$ such that for all  $n \geq N_2(x)$, 
    \begin{align}\label{eq:ass:riso}
      \Bigg(\sum_{y \in B(x,n)}\! |u(y)|^{\frac{d'}{d'-1}}\Bigg)^{\!\!\frac{d'-1}{d'}}
      \;\leq\;
      C_{\mathrm{S_1}}\, n^{1 - \frac{d}{d'}}\mspace{-6mu}
      \sum_{\substack{y \vee z \in B(x,n)\\ \{y,z\} \in E}}\mspace{-6mu}
      \big|u(y) - u(z) \big|
    \end{align}
    for all $u\! : V \to \bbR$ with $\supp u \subset B(x,n)$.
  \end{itemize}
\end{assumption}
\begin{remark}\label{rem:graph:Zd}
  The Euclidean lattice, $(\bbZ^d, E_d)$, satisfies the Assumption~\ref{ass:graph} with $d' = d$ and $N_1(x) = N_2(x) = 1$.
\end{remark}
\begin{remark}
  It was recently shown in \cite{Ng14}, that the infinite cluster of a supercritical Bernoulli percolation satisfies the Assumption~\ref{ass:graph} for some $d' > d$.
\end{remark}
\begin{remark}
  The following \emph{strong isoperimetric inequality for large balls} is sufficient for the local Sobolev inequality $(S^1_{d'})$ to hold.  That is, for all $n$ large enough,
  \begin{align}
    |\partial A|
    \;\geq\;
    C_{\mathrm{iso}}\, |A|^{(d-1)/d},
    \qquad
    \forall\, A \subset B(x,n)\; \text{ s.th. } |A| \geq n^{\th},
  \end{align}
  where $\th = (d'-d)/(d'-1)$, see \cite{DNS15}.
\end{remark}
Consider a family of positive weights $\om = \{\om(e) \in (0, \infty) : e \in E\}$.  With an abuse of notation we also denote the \emph{conductance matrix} by $\om$, that is for $x, y \in V$ we set $\om(x,y) = \om(y,x) = \om(\{x,y\})$ if $\{x,y\} \in E$ and $\om(x,y) = 0$ otherwise.  We also refer to $\om(e)$ as the \emph{conductance} of the edge $e$. Let us further define measures $\mu^{\om}$ and $\nu^{\om}$ on $V$ by
\begin{align*}
  \mu^{\om}(x) \;\ldef\; \sum_{y \sim x}\, \om(x,y)
  \qquad \text{and} \qquad
  \nu^{\om}(x) \;\ldef\; \sum_{y \sim x}\, \frac{1}{\om(x,y)}.
\end{align*}
For any fixed $\om$ we consider a continuous time Markov chain, $Y = (Y_t\!: t \geq 0 )$, on $V$ with generator $\cL^{\om} \equiv \cL^{\om}_{\mathrm{C}}$ acting on bounded functions $f\!:  V \to \bbR$ as
\begin{align} \label{def:LY}
  \big(\cL^{\om} f)(x)
  \;=\; \frac{1}{\mu^\om(x)}
  \sum_{y \sim x} \om(x,y) \, \big(f(y) - f(x)\big).
\end{align}
Let us stress the fact that the Markov chain, $Y$, is \emph{reversible} with respect to the measure $\mu^{\om}$.  Setting  $p^{\om}(x,y) \ldef \om(x,y) / \mu^{\om}(x)$, this stochastic process waits at $x$ an exponential time with mean $1$ and chooses its next position $y$ with probability $p^{\om}(x,y)$.  Since the law of the waiting times does not depend on the location, $Y$ is also called the \emph{constant speed random walk} (CSRW).  

Another natural choice for a random walk that jumps from $x$ to $y$ with probability $p^{\om}(x,y)$ is  the \emph{variable speed random walk} (VSRW) $X = (X_t\!: t \geq 0 )$, which waits at $x$ an exponential time with mean $1/\mu(x)$, with generator given by
\begin{align*}
  \big(\cL^{\om}_\mathrm{V} f)(x)
  \;=\;
  \sum_{y \sim x} \om(x,y)\, \big(f(y) - f(x)\big)
  \;=\;
  \mu^\om(x)\, \big(\cL^{\om}_{\mathrm{C}} f)(x).
\end{align*}
We recall that  the VSRW $X$ is reversible with respect to the counting measure and that the CSRW and the VSRW are time-changes of each other. More precisely, $Y_t = X_{a_t}$ for $t \geq 0$, where $a_t \ldef \inf\{s \geq 0 : A_s > t\}$ denotes the right continuous inverse of the functional
\begin{align*}
  A_t \;=\; \int_0^t \mu^\om(X_s) \, \md s, \qquad t \geq 0.
\end{align*}
We denote by $\Prob_x^{\om}$ the law of the process $X$ or $Y$, respectively, starting at the vertex $x \in V$. For $x, y \in V$ and $t \geq 0$ let
$q^{\om}(t,x,y)$ and $p^{\om}(t,x,y)$  be the transition densities of $Y$ and $X$ with respect to the reversible measures (or the \emph{heat kernels} associated with $\cL^{\om}_{\mathrm{C}}$ and $\cL^{\om}_{\mathrm{V}}$), i.e.\
\begin{align*}
  q^{\om}(t,x,y)
  \;\ldef\;
  \frac{\Prob_x^{\om}\big[Y_t = y\big]}{\mu^{\om}(y)}, \qquad p^{\om}(t,x,y)
  \;\ldef\;
  \Prob_x^{\om}\big[X_t = y\big].
\end{align*}
For any non-empty, finite $A \subset V$ and $p \in [1, \infty)$, we introduce space-averaged $\ell^{p}$-norms on functions $f\!: A \to \bbR$ by the usual formula
\begin{align*}
  \Norm{f}{p,A}
  \;\ldef\;
  \Bigg(
    \frac{1}{|A|}\; \sum_{x \in A}\, |f(x)|^p
  \Bigg)^{\!\!\frac{1}{p}}
  \qquad \text{and} \qquad
  \Norm{f}{\infty, A} \;\ldef\; \max_{x\in A} |f(x)|.
\end{align*}
Further, for any $x \in V$ we set
\begin{align*}
  \bar{\mu}_p(x)
  \;\ldef\;
  \limsup_{n\to \infty} \Norm{\mu^{\om}}{p,B(x,n)}
  \qquad \text{and} \qquad
  \bar{\nu}_q(x)
  \;\ldef\;
  \limsup_{n \to \infty} \Norm{\nu^{\om}}{q,B(x,n)}.
\end{align*}
For our main results we need to make the following assumption on the integrability of the conductances.
\begin{assumption} \label{ass:lim_const}
  There exist $p, q \in (1,\infty]$ with
  \begin{align}\label{eq:cond:pq}
    \frac{1}{p} \,+\, \frac{1}{q} \;<\; \frac{2}{d'}
  \end{align}
  such that
  \begin{align*}
    \bar{\mu}_p
    \;\ldef\;
    \sup_{x\in V} \bar{\mu}(x)
    \;<\;
    \infty
    \qquad \text{and} \qquad
    \bar{\nu}_q
    \;\ldef\;
    \sup_{x\in V} \bar{\nu}(x)
    \;<\;
    \infty.
  \end{align*}
  In particular, for every $x \in V$ there exists $N(x) \equiv N(x,\om)$ such that
  \begin{align*}
    \sup_{n \geq N(x)} \Norm{\mu^{\om}}{p,B(x,n)}
    \;\leq\;
    2 \bar{\mu}_p(x)
    \qquad \text{and} \qquad
    \sup_{n \geq N(x)} \Norm{\nu^{\om}}{q,B(x,n)}
    \;\leq\;
    2 \bar{\nu}_q(x).
  \end{align*}
\end{assumption} 
Our aim is to continue the program initiated in \cite{ADS13,ADS13a}.  In \cite{ADS13} we showed a quenched invariance principle for the random walks $X$ and $Y$ on the integer lattice $\bbZ^d$ under ergodic, degenerate random conductances satisfying a certain moment condition so that Assumption~\ref{ass:lim_const} is fulfilled (cf.\ Remark~\ref{rem:rcm} below).  In \cite{ADS13a} we established elliptic and parabolic Harnack inequalities for the generators, from which a local limit theorem for the heat kernel were deduced. In this paper we prove  Gaussian-type upper bound on the heat kernels. Let us first consider the CSRW.
\begin{theorem} \label{thm:hke}
  Suppose that Assumption~\ref{ass:lim_const} holds.  Then, there exist constants $c_i = c_i(d, p, q, \bar{\mu}_p, \bar{\nu}_q)$ such that for any given $t$ and $x$ with $\sqrt{t} \geq N(x) \vee 2(N_1(x) \vee N_2(x))$ and all $y \in V$ the following hold.
  \begin{enumerate}
  \item [(i)] If $d(x,y)\leq c_1 t$ then
    \begin{align*}
      q^\om(t,x, y)
      \;\leq\;
      c_2\, t^{-d/2}\,  \exp\!\big(\!-c_3\, d(x,y)^2/t\big).
    \end{align*}
  \item [(ii)] If $d(x,y)\geq c_1 t$ then
    \begin{align*}
      q^\om(t,x, y)
      \;\leq\;
      c_2 \, t^{-d/2} \exp\!\big(\!-c_4\, d(x,y) (1 \vee \log(d(x,y)/t))\big).
    \end{align*}
  \end{enumerate}
\end{theorem}
\begin{remark}
  Note that the moment condition on $\nu^{\om}$ can be improved by imposing a moment condition on $\nu^{\bar{\om}}$ for some weight $\bar{\om}$ provided the corresponding Dirichlet forms are comparable, that is
  \begin{align}\label{eq:DF:compare}
    \sum_{x,y \in V} \bar{\om}(x,y)\, \big(f(x) - f(y)\big)^2
    \;\leq\;
    \sum_{x,y \in V} \om(x,y)\, \big(f(x) - f(y)\big)^2.
  \end{align}
  In particular, notice that for any $t>0$ and
  \begin{align*}
    \bar{\om}(x,y) \;=\; \frac{1}{t} \mu^{\om}(x)\, q^{\om}(t,x,y)\, \mu^{\om}(y)
    \qquad \text{or} \qquad
    \bar{\om}(x,y) \;=\; \frac{1}{t} \mu^{\om}(x)\,p^{\om}(t,x,y)
  \end{align*}
  the inequality \eqref{eq:DF:compare} holds true.  Hence, it suffices that a moment condition on the heat kernel is satisfied, cf.\ \cite[Section~6]{ADS13a}.
\end{remark}
One well-established model which naturally fulfills our assumptions is the random conductance model on $\bbZ^d$.
\begin{remark} \label{rem:rcm}
  Consider the $d$-dimensional Euclidean lattice $\bbZ^d$ with $d \geq 2$, and let $E_d$  be the set of all non-oriented nearest neighbour bonds, i.e.\ $E_d \ldef \{ \{x,y\}: x,y \in \bbZ^d, |x-y|=1\}$.  Then, $(\bbZ^d, E_d)$ satisfies the Assumption \ref{ass:graph} as pointed out in Remark~\ref{rem:graph:Zd}.  Further, let $\prob$ be a probability measure on the measurable space $(\Om, \cF) = \big(\bbR_+^{E_d}, \cB(\bbR_+)^{\otimes\, E_d}\big)$  and write $\mean$ for the expectation with respect to $\prob$.  The space shift by $z \in \bbZ^d$ is the map $\tau_z\!: \Om \to \Om$ defined by $ (\tau_z \om)(x, y) \ldef \om(x+z, y+z)$ for all  $\{x,y\} \in E_d$.  Now assume that $\prob$ satisfies the following conditions.
  \begin{enumerate}[(i)]
  \item $\prob$ is ergodic with respect to translations of $\bbZ^d$, i.e. $\prob \circ\, \tau_x^{-1} \!= \prob\,$ for all $x \in \bbZ^d$ and $\prob[A] \in \{0,1\}\,$ for any $A \in \cF$ such that $\tau_x(A) = A\,$ for all $x \in \bbZ^d$.
  \item There exist $p,q \in (1, \infty]$ satisfying $1/p + 1/q < 2/d$ such that
    \begin{align} \label{eq:moment}
      \mean\!\big[\om(e)^p\big] \;<\; \infty
      \quad \text{and} \quad
      \mean\!\big[\om(e)^{-q}\big] \;<\; \infty
    \end{align}
    for any $e \in E_d$.
  \end{enumerate}
  Then, the spatial ergodic theorem gives that for $\prob$-a.e.\ $\om$,
  \begin{align*}
    \lim_{n \to \infty} \Norm{\mu^{\om}}{p, B(n)}^p
    \;=\;
    \mean\!\big[ \mu^{\om}(0)^p\big]%^{1/p}
    \;<\;
    \infty
    \quad \text{and} \quad
    \lim_{n\to \infty} \Norm{\nu^{\om}}{q, B(n)}^q
    \;=\;
    \mean\big[\nu^{\om}(0)^q\big]%^{1/q}
    \;<\;
    \infty.
  \end{align*}
  In particular, Assumption \ref{ass:lim_const} is fulfilled in this case and therefore for $\prob$-a.e.\ $\om$ the upper estimates on $q^\om_t(x,y)$ in Theorem~\ref{thm:hke} hold.  Unfortunately, we cannot  provide any control on the size of $\{N^{\om}(x): x \in V\}$ in the context of general ergodic environments, as we would need some information on the speed of convergence in the ergodic theorem, which is not available in this general framework unless we make additional mixing assumptions. 

  It has been been shown in \cite[Theorem 5.4]{ADS13a} that the moment condition in \eqref{eq:moment} is optimal for a local limit theorem to hold. In particular, this moment condition is also necessary for \emph{both} upper and lower Gaussian near-diagonal bounds to be satisfied. 
\end{remark}
Next we state the upper bounds on the heat kernel $p^{\om}(t, x ,y)$ of the VSRW. For that purpose we need to introduce the so called chemical distance $d^{\om}$ defined by
\begin{align*}
  d^{\om}(x,y)
  \;\ldef\;
  \inf_{\gamma}
  \Bigg\{
    \sum_{i=0}^{l_{\gamma}-1} 1 \wedge \om(z_i,z_{i+1})^{-1/2}
  \Bigg\},
\end{align*}
where the infimum is taken over all paths $\ga = (z_0, \ldots, z_{l_\gamma})$  connecting $x$ and $y$.  Note that $d^{\om}$ is a metric which is adapted to the transition rates of the random walk.  In particular, $d^{\om}(x,y) \leq d(x,y)$ for all $x,y \in V$.

% We denote by $\tilde{B}(x,r)$ the closed ball with center $x$ and radius $r$ with respect to $d^{\om}$, that is $\tilde{B}(x,r) \ldef \{y \in V \mid d^{\om}(x,y) \leq r\}$.  Notice that $d^{\om}(x,y) \leq d(x,y)$ for all $x,y \in V$ and therefore $B(x,r) \subseteq \tilde{B}(x,r)$ for all $x\in V$ and $r > 0$.  Moreover, for any $x \in V$ we define 
% \textcolor{red}{
% %
% \begin{align*}
%   \tilde{\mu}_p(x)
%   \;\ldef\;
%   \limsup_{n\to \infty} \Norm{\mu^{\om}}{p,\tilde{B}(x,n)}
%   \quad \text{and} \quad
%   \tilde{\nu}_q(x)
%   \;\ldef\;
%   \limsup_{n \to \infty} \Norm{\nu^{\om}}{q,\tilde{B}(x,n)}.
% \end{align*}
% % 
For the VSRW we impose the following assumption on the vertex degree of the underlying graph.
\begin{assumption}\label{ass:lim_const:VSRW}
  There exists $C_{\text{deg}} < [1,\infty)$ such that
  \begin{align}
    |\{y \,:\, y \sim x\}| \;\leq\; C_{\text{deg}}
    \qquad \forall\, x \in V.
  \end{align}
\end{assumption}
\begin{theorem} \label{thm:hke_VSRW}
  Suppose that Assumption~\ref{ass:lim_const} and \ref{ass:lim_const:VSRW} holds.  Then, there exist constants $c_i = c_i(d, p, q, \bar{\mu}_p, \bar{\nu}_q)$ and $\ga = \ga(d,p,q)$ such that for any given $t$ and $x$ with $\sqrt{t} \geq N(x) \vee 2(N_1(x) \vee N_2(x))$ and all $y \in V$ the following hold.
  \begin{enumerate}
  \item [(i)] If $d^{\om}(x,y)\leq c_5 t$ then
    \begin{align*}
      p^{\om}(t,x, y)
      \;\leq\;
      c_6\, t^{-d/2}\, \bigg( 1+\frac{d(x,y)}{\sqrt{t}} \bigg)^{\!\!\ga}\,
      \exp\!\big(\!-c_7\, d^{\om}(x,y)^2/t\big).
    \end{align*}
  \item [(ii)] If $d^\om(x,y)\geq c_5 t$ then
    \begin{align*}
      p^{\om}(t,x, y)
      \;\leq\;
      c_6 \, t^{-d/2}\, \bigg( 1+\frac{d(x,y)}{\sqrt{t}} \bigg)^{\!\!\ga}\,
      \exp\!\bigg(%
        \!-c_8\, d^{\om}(x,y) \bigg(1 \vee \log\frac{d^{\om}(x,y)}{t}\bigg)
      \bigg).
    \end{align*}
  \end{enumerate}
\end{theorem}
As already mentioned in the beginning, for random walks on weighted graphs Gaussian type estimates on the heat kernel have been proven by Delmotte \cite{De99} in the case, where the conductances are uniformly elliptic, i.e.\ $c^{-1} \leq \om (e) \leq c$, $e\in E$, for some $c \geq 1$. However, Gaussian bounds do not hold in general as under i.i.d.\ conductances with fat tails at zero the heat kernel decay may be sub-diffusive due to a trapping phenomenon -- see \cite{BBHK08, BB12}. 

On the other hand, in the symmetric setting it is well known that on-diagonal estimates are equivalent to a Nash inequality of the type
\begin{align*}
  \Bigg(\sum_{x \in V}\; |f(x)|^{2}\Bigg)^{\!\!1 + 2/d'}
  \!\leq\;
  C_{\mathrm{Nash}}\,
  \Bigg(\sum_{x,y \in V}\om(x,y)\, \big( f(x) - f(y) \big)^2\Bigg)
  \Bigg(\sum_{x \in V}\; |f(x)| \Bigg)^{\!\!4/d'}\mspace{-18mu},
\end{align*}
see \cite[Theorem 2.1]{CKS87}.  In particular, such a Nash inequality holds on $\bbZ^d$ with $d' = d$ and thus for the conductance model with conductances that are uniformly bounded away from zero.  However, in the general unbounded case no such Nash inequality is available.  Recently, an anchored version of the Nash inequality has been introduced by Mourrat and Otto in \cite{MO15} for a random conductance model on $\bbZ^d$ with conductances unbounded from below but bounded from above.  In particular, they obtain on-diagonal upper bounds under a suitable moment condition that are very similar to the one stated in Assumption~\ref{ass:lim_const}.  Remarkably, their results extend to degenerate time-dependent conductances.

For the VSRW with conductances that are only uniformly bounded away from zero, Gaussian off-diagonal bounds have been proven in \cite[Theorems 2.19 and 3.3] {BD10}.  In this setting, upper bounds have also been obtained in \cite[Theorem 10.1]{Mo09}.  However, the distance function that appears in the upper estimates for the VSRW in \cite{BD10,Fo11,Mo09} and in Theorem~\ref{thm:hke_VSRW} above is the chemical distance which can be quite different from the graph metric as the following example shows.
\begin{example}
  Let $\{Z_k : k \in \bbZ \}$ be a collection of i.i.d.\ random variables on a probability space $(\Om, \cF, \prob)$, taking values in $[1, \infty)$ with tail behaviour $\prob\big[ Z_1 > u \big] \sim u^{-\al}$ as $u \to \infty$ for $\al > 1$.  We fix a constant $c_Z > 0$ and $u_0 > 0$ such that $\prob\big[ Z_1 > u \big]= (c_Z u)^{-\alpha}$ for all $u \geq u_0$.  As underlying graph we take the two-dimensional Euclidean lattice $\bbZ^2$. Let $e_1$, $e_2$ be the canonical basis vectors of $\bbR^2$ and for any $x \in \bbZ^2$ we write $x^i$, $i = 1,2$, for its coordinates.  Consider an ergodic environment of random conductances defined by
  \begin{align*}
    \om(x,y)
    \;\ldef\;
    \begin{cases}
      1,            & \text{if } x-y = \pm e_2,\\
      Z_{x^2}, \quad & \text{if } x-y = \pm e_1. 
    \end{cases}
  \end{align*}
  That is, all edges in vertical direction (meaning $e_2$ direction) have conductance $1$, while the conductances on edges in horizontal direction are random, constant along each line, but independent between different lines.  Note that this example can be easily generalised to arbitrary dimensions $d \geq 2$.
\end{example}
In this example the chemical distance becomes much smaller than the Euclidean distance as stated in the following lemma whose proof will be given in Appendix~\ref{sec:chemical:distance}.
\begin{lemma}\label{lemma:chemical:distance}
  For $\de \in (0,1)$ and $\prob$-a.e.\ $\om$ there exists $0 < c_Z(\al, \de) < C_Z(\al, \de) < \infty$ and $L_0 = L_0(\om) < \infty$ such that for all $L \geq L_0$,
  \begin{align}
    c_Z\, (\ln L)^{-(1+\de)/(2\al+1)}\,L^{2\al / (2\al+1)}
    \;\leq\;
    d^{\om}(0, L e_1)
    \;\leq\;
    C_Z(\ln L)^{\de/2\al}\,L^{2\al / (2\al+1)}.
  \end{align}
\end{lemma}
On the other hand, it is shown in \cite{BD10} that in the case of i.i.d.\ conductances the chemical distance $d^{\om}(x,y)$ can be compared with the graph distance $d(x,y)$ provided that $d(x,y)$ is large enough.
\begin{remark}
  In view of  Theorem \ref{thm:hke_VSRW} (i) for the choice $t=L$ one would get an upper bound  given by
  \begin{align*}
    p^\omega_t(0, t e_1)
    \;\leq\;
    C\, \exp(-c\, t^{(2\al-1 - \ve)/(2\al+1)})
  \end{align*}
  for any $\ve > 0$ and $L > L_0(\om)$.  However, this estimate is not optimal and does not match with a lower bound.  The correct order is of the form 
  \begin{align*}
    c \exp(-c\, t^{(3\al - 1 + \ve)/(3\al + 1)})
    \;\leq\;
    p^{\om}_t(0, t e_1)
    \;\leq\;
    C \exp(-c\, t^{(3\al - 1 - \ve)/(3\al + 1)}),
  \end{align*}
  which will be proven in an upcoming paper by the second author and R.\ Fukushima.
\end{remark}
Clearly, one would also like to establish corresponding lower bounds.  It is well known that Gaussian lower and upper bounds on the heat kernel are equivalent to a parabolic Harnack inequality in many situations, for instance in the case of uniformly elliptic conductances, see \cite{De99}.  Recently, this equivalence has also been established on locally irregular graphs in \cite{BC14}.  In our context such a parabolic Harnack inequality has been recently proven in \cite{ADS13a}.  Unfortunately, due to the special structure of the Harnack constant in \cite{ADS13a}, in particular its dependence on $\|\mu^{\om}\|_{p,B(x,n)}$ and $\|\nu^{\om}\|_{q,B(x,n)}$, we cannot directly deduce off-diagonal Gaussian lower bounds from it.  More precisely, in order to get effective Gaussian off-diagonal bounds using the established chaining argument (see e.g.\ \cite{Ba04}), one needs to apply the Harnack inequality on a number of balls with radius $n$ having a distance of order $n^2$.  In general, the ergodic theorem does not 
give the required uniform control on the convergence of space-averages of stationary random variables over such balls (see \cite{AJ75}).

Moreover, in the setting of VSRW with conductances unbounded from above but uniformly bounded from below the chaining technique would yield off-diagonal Gaussian lower bound with respect to the usual graph metric $d$ instead of the chemical distance $d^{\om}$.  Therefore, the problem to find matching upper and lower off-diagonal Gaussian bounds for general random conductance models remains open.

\subsection{The method}
A technique that turned out to perform quite well in order to prove the Gaussian upper bound in Theorem~\ref{thm:hke}, is known as \emph{Davies' method} in the literature (see e.g.\ \cite{Da89,Da93, CKS87}).  In contrast to the chaining argument mentioned above the main advantage of Davies' technique is that we only need to apply the ergodic theorem (or Assumption~\ref{ass:lim_const}, respectively) on balls with one fixed center point $x_0$.

We now briefly sketch the idea of Davies' method. Instead of studying the original semigroup $(P_t: t \geq 0)$ which is generating the random walk $Y$, that is
\begin{align*}
  \big(P_t f\big)(x)
  \;=\;
  \sum_{y \in V} \mu^\om(y)\, q^\om(t,x,y)\, f(y),
\end{align*}
Davies suggests to consider the semigroup $( P_t^\psi: t \geq 0 )$ given by
\begin{align*}
  \big(P_t^{\psi} f\big)(x)
  \;=\;
  \me^{\psi(x)}\, \big(P_t(e^{-\psi}f)\big)(x),
\end{align*}
with generator
\begin{align*}
  \big(\cL^{\psi} f\big)(x)
  \;=\;
  \me^{\psi(x)} \big(\cL^{\om}(e^{-\psi}f) \big)(x),
\end{align*}
for a suitable class of test functions $\psi$. Clearly, this semigroup has a kernel which is given by  $\me^{\psi(x)} q^\om(t,x,y)\, \me^{-\psi(y)}$ and satisfies the heat equation  $\partial_t v - \cL^\psi v=0$. Note that $P_t^{\psi}$ is symmetric with respect to the measure $\me^{-2\psi} \mu^{\om}$.

In the classical setting of symmetric Markov semigroups whose generator is a second order elliptic operator, the Nash inequality and equivalently Gaussian on-diagonal estimates do hold. Then, Davies used the classical Leibniz rule to derive a bound on the kernel of $(P_t^\psi : t \geq 0)$, which can be rewritten as
\begin{align*}
  q(t,x,y)
  \;\leq\;
  c\, t^{-d/2}\, \me^{\psi(y)-\psi(x)+t \Ga(\psi)},
\end{align*}
where $\Ga$ denotes the carr\'{e} du champ operator. Finally, by varying over $\psi$ Gaussian upper bounds can be obtained. For further details we refer to \cite{CKS87}. The method has also been used to obtain the Gaussian upper bounds in \cite{De99}.

In our setting, where the conductances are unbounded from below, we do not have a Nash inequality available. Therefore, we follow an approach used by Zhikov in \cite{Zi13}, where some upper bounds for the solution kernel of certain degenerate Cauchy problems on $\bbR^d$ are obtained. More precisely, we use Moser's iteration technique to show a maximal inequality for solution of $\partial_t v - \cL^{\psi} v = 0$ and combine it with Davies' method. Similarly to \cite{CKS87}, where Davies' method has been carried out for processes generated by non-local Dirichlet forms, one difficulty is the absence of a Leibniz rule in the discrete setting of a graph.  In \cite{ADS13a} we already established a Moser iteration scheme and a maximal inequality for solutions of the original heat equation $\partial_t u - \cL^{\om} u = 0$, so we adapt here the arguments in \cite{ADS13a} to deal with the perturbed semigroup $(P_t^\psi : t \geq 0)$.    

The rest of the paper is organised as follows. In Section~\ref{sec:GBcsrw} we prove Theorem~\ref{thm:hke} and in Section~\ref{sec:GBvsrw} we explain how the proof of Theorem~\ref{thm:hke} needs to be modified in order to obtain Theorem~\ref{thm:hke_VSRW}. The appendix contains the proof of Lemma~\ref{lemma:chemical:distance} as well as a collection of some elementary estimates needed in the proofs.  Throughout the paper we write $c$ to denote a positive constant which may change on each appearance. Constants denoted $C_i$ will be the same through each argument.

\section{Gaussian upper bounds for the CSRW} \label{sec:GBcsrw}
This section is devoted to the proof of Theorem~\ref{thm:hke}.  It is convenient to introduce a potential theoretic setup.  First of all, for $f\!: V \to \bbR$ and $F\!: E \to \bbR$ we define the operators $\nabla f\!: E \to \bbR$ and $\nabla^*F\!: V \to \bbR$ by
\begin{align*}
  \nabla f(e) \;\ldef\; f(e^+) - f(e^-),
  \qquad \text{and} \qquad
  \nabla^*F (x)
  \;\ldef\;
  \sum_{e: e^+ =\,x}\! F(e) \,-\! \sum_{e:e^-=\, x}\! F(e),
\end{align*}
where for each non-oriented edge $e \in E$ we specify one of its two endpoints as its initial vertex $e^+$ and the other one as its terminal vertex $e^-$.  Nothing of what will follow depend on the particular choice.  Since for all $f \in \ell^2(V)$ and $F \in \ell^2(E)$ it holds that $\scpr{\nabla f}{F}{E} = \scpr{f}{\nabla^* F}{V}$, $\nabla^*$ can be seen as the adjoint of $\nabla$.  For $f\!: V \to \bbR$ and $F\!: E \to \bbR$ we also define the products $f \cdot F$ and $F \cdot f$ by
\begin{align*}
  \big(f \cdot F\big)(e) \;\ldef\; f(e^-)\, F(e)
  \qquad \text{and} \qquad
  \big(F \cdot f\big)(e) \;\ldef\; f(e^+)\, F(e).
\end{align*}
Then, the discrete analog of the product rule can be written as
\begin{align}\label{eq:rule:prod}
  \nabla(f g)
  \;=\;
  \big(g \cdot \nabla f\big) \,+\, \big(\nabla g \cdot f\big)
  \;=\;
  \av{f} \nabla g \,+\, \av{g} \nabla f,
\end{align}
where $\av{f}(e) \ldef \frac{1}{2}(f(e^+) + f(e^-))$.  On the weighted Hilbert space $\ell^2(V, \mu^{\om})$ the \emph{Dirichlet form} or \emph{energy} associated to $\cL^{\om}$ is given by
\begin{align} \label{eq:def:dform}
  \cE^{\om}(f,g)
  \;\ldef\;
  \scpr{f}{-\cL^{\om} g}{V, \mu^\om}
  \;=\;
  \scpr{\nabla f}{\om \nabla g}{E}
  \;=\;
  \scpr{1}{\md \Ga^{\om}(f,g)}{E},
\end{align}
where $\md \Ga^{\om}(f,g) \ldef \om \nabla f \nabla g$ and we set $\cE^{\om}(f) \ldef \cE^{\om}(f,f)$.  For a given function $\eta\!: B \subset V \to \bbR$, we denote by $\cE_{\eta^2}^{\om}(u)$ the Dirichlet form where $\om(e)$ is replaced by $\av{\eta^2}\, \om(e)$ for $e \in E$.

\subsection{A-priori estimate for the perturbed Cauchy problem}
We consider now the Cauchy problem
\begin{align}\label{eq:cauchy_prob}
   \left\{
    \begin{array}{rcl}
      \partial_t u - \cL^{\om} u
      &\mspace{-5mu}=\mspace{-5mu}& 0,
      \\[1ex]
      u(t=0, \,\cdot\,)
      &\mspace{-5mu}=\mspace{-5mu}& f,
    \end{array}
  \right.
\end{align}
and for any positive function $\phi$ on $V$ such that $\phi, \phi^{-1} \in \ell^{\infty}(V)$ we define the operator $\cL_{\phi}^{\om}$ acting on bounded functions $g\!: V \to \bbR$ as
\begin{align*}
  (\cL_{\phi}^{\om}\, g)(x)
  \;\ldef\;
  \phi(x) (\cL^{\om} \phi^{-1} g)(x).
\end{align*}
\begin{lemma} \label{lem:apriori}
  Suppose that $f \in \ell^2(V, \mu^{\om})$ and $u$ solves the corresponding Cauchy problem \eqref{eq:cauchy_prob}.  Further, set $v(t,x) \ldef \phi(x) u(t,x)$ for a positive function $\phi$ on $V$ such that $\phi, \phi^{-1} \in \ell^{\infty}(V)$.  Then
  \begin{align}\label{eq:hke:apriori}
    \norm{v(t, \,\cdot\,)}{2}{V, \mu^{\om}}
    \;\leq\;
    \me^{h(\phi) t}\, \norm{\phi f}{2}{V, \mu^{\om}},
  \end{align}
  where
  \begin{align*}
    h^{\om}(\phi)
    \;\ldef\;
    \max_{x \in V}
    \frac{1}{\mu^{\om}(x)}\,
    \sum_{y \sim x} \big| \md\Ga^{\om}(\phi, \phi^{-1})(\{x,y\}) \big|.
    % \bigg(
    %   \frac{\phi(e^+)}{\phi(e^-)} \,+\, \frac{\phi(e^-)}{\phi(e^+)} \,-\, 2
    % \bigg).
  \end{align*}
\end{lemma}
\begin{proof}
  We start by observing that the function $v$, defined above, solves the Cauchy problem $\partial_t v - \cL_{\phi}^{\om}\, v = 0$ with initial condition $v(0, \cdot) = \phi f$.  As a first step, we show that for all $g \in \ell^{2}(V, \mu^{\om})$,
  \begin{align}\label{eq:hypo:L}
    \scpr{g}{-\cL_{\phi}^{\om}\, g}{V, \mu^{\om}}
    \;\geq\;
    -h(\phi)\, \norm{g}{2}{V, \mu^{\om}}^2.
  \end{align}
  But, an application of the product rule \eqref{eq:rule:prod} yields 
  \begin{align} \label{eq:decomp_av}
    \scpr{g}{-\cL_{\phi}^{\om}\, g}{V, \mu^{\om}}
    &\;=\;
    \scpr{\nabla(\phi\, g)}{\om\, \nabla(\phi^{-1} g)}{E}
    \nonumber\\[.5ex]
    &\;=\;
    \scpr{\av{\phi} \av{\phi^{-1}}}{\md\Ga^{\om}(g, g)}{E}
    \,+\,
    \scpr{\av{g}^2}{\md\Ga^{\om}(\phi, \phi^{-1})}{E}
    \nonumber\\
    &\mspace{36mu}+\,
    \scpr{\av{g}}
    {
      \av{\phi} \md\Ga^{\om}(g, \phi^{-1})
      \,+\,
      \av{\phi^{-1}} \md\Ga^{\om}(\phi,g)
    }{E}
    % \scpr{\av{\phi} \av{g}}{\md\Ga^{\om}(g, \phi^{-1})}{E}
    % \,+\,
    % \scpr{\av{\phi^{-1}} \av{g}}{\md\Ga^{\om}(\phi,g)}{E}
    \nonumber\\[.5ex]
    &\;\geq\;
    \scpr{\av{g}^2}{\md\Ga^{\om}(\phi, \phi^{-1})}{E},
  \end{align}
  where we exploit the fact that, for any $e \in E$,
  \begin{align}\label{eq:phi:property}
    \begin{array}{c}
      \av{\phi^{-1}}(e) (\nabla \phi)(e)
      \;=\;
      \dfrac{1}{2}\,
      \bigg(
        \dfrac{\phi(e^+)}{\phi(e^-)} \,-\, \dfrac{\phi(e^-)}{\phi(e^+)}
      \bigg)
      \;=\;
      -\av{\phi}(e) (\nabla \phi^{-1})(e),
      \\[2.5ex]
      \av{\phi^{-1}}(e) \av{\phi}(e)
      \;=\;
      \dfrac{1}{4}\,
      \bigg(
        \dfrac{\phi(e^+)}{\phi(e^-)} \,+\, \dfrac{\phi(e^-)}{\phi(e^+)} \,+\, 2
      \bigg)
      \;\geq\;
      1.
    \end{array}
  \end{align}
  Note that $(\nabla \phi)(\nabla \phi^{-1}) \leq 0$. Then, $\av{g}^2 \leq \av{g^2}$ by Jensen's inequality, the claim \eqref{eq:hypo:L} follows.  Thus, setting $v_t(x) \ldef v(t,x)$, we have for any $t \geq 0$,
  \begin{align*}
    \partial_t \norm{v_t}{2}{V, \mu^{\om}}^2
    \;=\;
    2\, \scpr{v_t}{\cL_{\phi}^{\om} v_t}{V, \mu^{\om}}
    \;\stackrel{\eqref{eq:hypo:L}}{\leq}\;
    2\, h^{\om}(\phi)\, \norm{v_t}{2}{V, \mu^{\om}}^{2}.
  \end{align*}
  Solving this differential inequality gives immediately \eqref{eq:hke:apriori}.
\end{proof}

\subsection{Maximal inequality for the perturbed Cauchy problem}
Our next aim is to derive a maximal inequality for the function $v$. For that purpose we will adapt the arguments given in \cite[Section 4]{ADS13a} and set up a Moser iteration scheme.  For any non-empty, finite $B \subset V$ and $p \in [1, \infty)$, we introduce a space-averaged norm on functions $f\!: B \to \bbR$ by
\begin{align*}
  \Norm{f}{p,B,\mu^{\om}}
  \;\ldef\;
  \Bigg(
    \frac{1}{|B|}\, \sum_{x \in B} \mu^{\om}(x)\, |f(x)|^p
  \Bigg)^{\!\!1/p}.
\end{align*}
\begin{lemma} \label{lem:moser_pre}
  Suppose that $Q = I \times B$, where $I = [s_1, s_2] \subset \bbR$ is an interval and $B \subset V$ is finite and connected, and consider a function $\eta$ on $V$ with
  \begin{align*}
    \supp \eta \;\subset\; B, \qquad
    0 \;\leq\; \eta \;\leq\; 1 \qquad \text{and} \qquad
    \eta \equiv 0 \quad \text{on} \quad \partial B.
  \end{align*}
  For a given $\phi > 0$ with $\phi, \phi^{-1} \in \ell^{\infty}(V)$, let $v_t \geq 0$ be a solution of $\partial_t v - \cL_{\phi}^{\om} v \leq 0$ on $Q$.  Then, there exists $C_1 < \infty$ such that for all $\al \geq 1$, 
  \begin{align} \label{eq:DF:ode}
    \frac{\md}{\md t}\, \Norm{(\eta v_t^{\al})}{2, B, \mu^{\om}}^2
    \,+\,
    \frac{\cE_{\eta^2}^{\om}(v_t^{\al})}{|B|}
    \;\leq\;
    C_1 \al^2\,
    \Big(\norm{\nabla \eta}{\infty}{E}^2 + h^{\om}(\phi) \Big)\,
    \Norm{v_t^{\al}}{2, B, \mu^{\om}}^2.
  \end{align}
\end{lemma}
\begin{proof}
  Since $\partial_t v - \cL_{\phi}^{\om}v \leq 0$ on $Q$ we have for every $t \in I$ and $\al \geq 1$,
  \begin{align}\label{eq:initial:step}
    \frac{1}{2 \al}\,
    \frac{\md}{\md t}\, \norm{(\eta v_t^{\al})}{2}{V, \mu^{\om}}^2
    % \;=\;
    % \scpr{\eta^2 v_t^{2 \al -1}}{\partial_t v_t}{V, \mu^{\om}}
    &\;\leq\;
    \scpr{\eta^2 v_t^{2 \al-1}}{\cL_{\phi}^{\om} v_t}{V, \mu^{\om}}
    \nonumber\\[.5ex]
    &\;=\;
    - \scpr{\nabla (\eta^2 \phi\, v_t^{2 \al-1})}{\om \nabla (\phi^{-1} v_t)}{E}.
  \end{align}
  By applying the product rule \eqref{eq:rule:prod}, we obtain
  \begin{align*}
%    &\scpr{\eta^2 v_t^{2 \al-1}}{-\cL_{\phi}^{\om} v_t}{V,\mu^{\om}}
%    \\[.5ex]
%    &\mspace{36mu}=\;
    &\scpr{\nabla (\eta^2 \phi\, v_t^{2 \al-1})}{\om \nabla (\phi^{-1} v_t)}{E}
    \nonumber\\[.5ex]
    &\mspace{36mu}=\;
    \scpr{\av{\eta^2}}{\md \Ga^{\om}(\phi\, v_t^{2 \al -1}, \phi^{-1} v_t)}{E}
    \,+\,
    \scpr{\av{\phi\, v_t^{2 \al -1}}}{\md \Ga^{\om}(\eta^2, \phi^{-1} v_t)}{E}
    \\[.5ex]
    &\mspace{36mu}\rdef\;
    T_1 \,+\, T_2.
   % \quad (\text{say}).
  \end{align*}
  First, we derive a lower bound for $T_1$. Recall that $\md\Ga^{\om}(\phi, \phi^{-1}) \leq 0$.  In view of \eqref{eq:phi:property}, an expansion of $\md \Ga^{\om}(\phi\, v_t^{2\al-1}, \phi^{-1}v_t)$ by means of the product rule \eqref{eq:rule:prod} yields
  \begin{align}\label{eq:estimate:T1}
    \md \Ga^{\om}(\phi\, v_t^{2\al-1}, \phi^{-1} v_t)
    &\;\geq\;
    \av{\phi} \av{\phi^{-1}}\,\frac{2\al-1}{\al^2}\, 
    \md \Ga^{\om}(v_t^{\al}, v_t^{\al}) 
    \,+\,
    \av{v_t^{2\al}}\, \md \Ga^{\om}(\phi, \phi^{-1})
    \nonumber\\[.5ex]
    &\mspace{36mu}-\,
    \frac{2(\al-1)}{\al}\, \av{v_t^{\al}} \av{\phi}\,
    \big|\md \Ga^{\om}(v_t^{\al}, \phi^{-1})\big|,
  \end{align}
  where we used the fact that for any $\al > 1/2$,
  \begin{align*} 
    \md \Ga^{\om}(v_t^{2\al-1}, v_t)
    &\overset{\eqref{eq:rule:chain:1}}{\;\geq \;}
    \frac{2\al-1}{\al}\, \md\Ga^{\om}(v_t^{\al}, v_t^{\al}), 
  \end{align*}
  and that by H\"older's inequality, $\av{v_t^{\al_1}} \av{v_t^{\al_2}} \leq \av{v_t^{\al_1+\al_2}}$ for any $\al_1, \al_2 \geq 0$.  Moreover, we used that
  \begin{align*}
    &\big|
      \av{v_t}(e) \nabla v_t^{2\al-1}(e) - \av{v_t^{2\al-1}}(e) \nabla v_t(e)
    \big|
    \\
    &\mspace{36mu}=\;
    \big|v_t^{2\al -1}(e^+) v_t(e^-) - v_t^{2\al-1}(e^-) v_t(e^+) \big|
    \;\stackrel{\eqref{eq:rule:chain:2}}{\leq}\;
    \frac{2(\al-1)}{\al}\, \big| \av{v_t^{\al}}(e)\, \nabla v_t^{\al}(e) \big|
  \end{align*}
  for all $e \in E$.  In view of \eqref{eq:phi:property} note that
  \begin{align}\label{eq:av:nabla:phi}
    \av{\phi} \big| \nabla \phi^{-1} \big|
    \;=\;
    \sqrt{\av{\phi} \av{\phi^{-1}}} 
    \cdot \sqrt{-(\nabla \phi) (\nabla \phi^{-1})}.
  \end{align}
  Thus, an application of Young's inequality, that reads $|a b| \leq \frac{1}{2} ( a^2/ \ve + \ve\, b^2 )$, with $\ve = 4 (\al-1)$ to the last term in \eqref{eq:estimate:T1} results in 
  \begin{align*}
    T_1
    \;\geq\;
    %\bigg(\frac{2 \al -1}{\al^2} - \frac{1}{4 \al} \bigg)\,
    \frac{3}{4\al}\,
    \scpr{\av{\eta^2} \av{\phi} \av{\phi^{-1}}}
    {\md\Ga^{\om}(v_t^{\al}, v_t^{\al})}{E}
    \,-\,
    4\al\, h^{\om}(\phi)\, |B|\,
    \Norm{v_t^{\al}}{2, B, \mu^{\om}}^2.
  \end{align*}
  Next we consider the term $T_2$.  Since $\av{\phi\, v_t^{2\al-1}} \leq 2 \av{\phi} \av{v_t^{2\al-1}}$, we obtain by expanding $\md \Ga^{\om}(\eta^2, \phi^{-1} v_t)$ by means of the product rule \eqref{eq:rule:prod} that
  \begin{align}\label{eq:estimate:T2}
    \av{\phi\, v_t^{2\al-1}}\,\md \Ga^{\om}(\eta^2, \phi^{-1} v_t)
    &\;\geq\;
    -\, 16 \av{\phi} \av{\phi^{-1}} \av{\eta} \av{v_t^{\al}}\,
    \big|\md \Ga^{\om}(\eta, v_t^{\al})\big|
    \nonumber\\[.5ex]
    &\mspace{36mu}-\,
    4 \av{v^{2\al}} \av{\phi} \av{\eta}\,
    \big|\md \Ga^{\om}(\eta, \phi^{-1})\big|,
  \end{align}
  where we used the fact that for any $\al \geq 1/2$,
  \begin{align*}
    \av{v_t^{2\al-1}} \big|\md \Ga^{\om}(\eta, v_t)\big|
    \;\stackrel{\eqref{eq:rule:chain:3}}{\leq}\;
    4\, \av{v_t^{\al}} \big|\md \Ga^{\om}(\eta, v_t^{\al})\big|.
  \end{align*}
  Thus, by applying Young's inequality with $\ve = 16 \al$ to the first term and \eqref{eq:av:nabla:phi} together with Young's inequality with $\ve = 1/2$ to the second term on the right-hand side of \eqref{eq:estimate:T2} we obtain that for any $\al \geq 1$,
  \begin{align*}
    T_2
    &\;\geq\;
    -\frac{1}{4\al}\,
    \scpr{\av{\eta^2} \av{\phi} \av{\phi^{-1}}}
    {\md\Ga^{\om}(v_t^{\al}, v_t^{\al})}{E}
    \\
    &\mspace{36mu}-\,
    \scpr{\av{\eta^2}\av{v_t^{2\al}}}{|\md\Ga^{\om}(\phi, \phi^{-1})|}{E}
    \\
    &\mspace{36mu}-\,
    260 \al\,
    \scpr{\av{\phi}\av{\phi^{-1}}\av{v_t^{2\al}}}{\md\Ga^{\om}(\eta,\eta)}{E}
    \\[.5ex]
    &\;\geq\;
    -\frac{1}{4\al}\,
    \scpr{\av{\eta^2} \av{\phi} \av{\phi^{-1}}}
    {\md\Ga^{\om}(v_t^{\al}, v_t^{\al})}{E}
    \\
    &\mspace{36mu}-\al\, 
    \Big(
      260\, \norm{\nabla \eta}{\infty}{E}^2 - 66\, h^{\om}(\phi)
    \Big)\, |B|\, \Norm{v_t^{\al}}{2, B, \mu^{\om}}^2,
  \end{align*}
  % \begin{align*}
  %   T_2
  %   &\;\geq\;
  %   -\frac{1}{4\al}\, \cE_{\eta^2}^{\om}(v_t^{\al})
  %   \,-\,
  %   \Big(
  %     64\,\al\, A(\phi)^2\, \norm{\nabla \eta}{\infty}{E}^2 
  %     + 2\, A(\phi)\, \norm{\nabla \eta}{\infty}{E}
  %   \Big)\, |B|\,
  %   \Norm{v_t^{\al}}{2, B, \mu^{\om}}^2
  %   \\[.5ex]
  %   &\;\geq\;
  %   -\frac{1}{4\al}\, \cE_{\eta^2}^{\om}(v_t^{\al})
  %   \,-\,
  %   65\, \al\, A(\phi)^2\, \Big(1 + \norm{\nabla \eta}{\infty}{E}^2\Big)\,
  %   |B|\,\Norm{v_t^{\al}}{2, B, \mu^{\om}}^2,
  % \end{align*}
  %
  where we used that $\av{\phi}\, \av{\phi^{-1}} = 1 - \frac{1}{4}(\nabla \phi)(\nabla \phi^{-1})$.  Hence, in view of \eqref{eq:phi:property}, there exists a constant $C_1 < \infty$ such that
  \begin{align} \label{eq:est_T12}
    T_1 + T_2
    \;\geq\;
    \frac{1}{2\al}\, \cE_{\eta^2}^{\om}(v_t^{\al}, v_t^{\al})
    \,-\,
    \frac{C_1}{2}\, \al\,
    \Big(\norm{\nabla \eta}{\infty}{E}^2 + h^{\om}(\phi) \Big)\,
    |B|\, \Norm{v_t^{\al}}{2, B, \mu^{\om}}^2.
  \end{align}
  In view of \eqref{eq:initial:step} the assertion follows.
\end{proof}
As an easy consequence we obtain now the analogue to \cite[Lemma 4.1]{ADS13a}.
\begin{corro} \label{cor:prep_moser}
  Under the assumptions of Lemma~\ref{lem:moser_pre} consider a function $\ze\!: \bbR \to \bbR$ satisfying $\supp \ze \subset I$, $0 \;\leq\; \ze \;\leq\; 1$ and $\ze(s_1) \;=\; 0$.  Then, for all $\al \geq 1$,
  \begin{align}\label{eq:DF}
    &\int_{I} \ze(t)\; \frac{\cE_{\eta^2}^{\om}(v_t^{\al})}{|B|} \, \md t
    \;\leq\;
    C_1 \al^2\,
    \Big( 
      \norm{\nabla \eta}{\infty}{\!E}^2 
      +
      \big\| \ze' \big\|_{\raisebox{-0ex}{$\scriptstyle L^{\raisebox{.1ex}{$\scriptscriptstyle \!\infty$}} (I)$}}
      +
      h^{\om}(\phi)
    \Big)\; \int_{I}\; \Norm{v_t^{2 \al}}{1, B, \mu^{\om}}\, \md t
    \nonumber\\[-.5ex]
  \end{align}
  \vspace{-4ex}

  \noindent and
  \begin{align}\label{eq:max}
    & \max_{t \in I}
    \Big( \ze(t)\; \Norm{(\eta\, v_t^{\al})^2}{1, B, \mu^{\om}} \Big)
    \nonumber\\[.5ex]
    &\mspace{36mu}
    \;\leq\;  
    C_1 \al^2\,
    \Big( 
      \norm{\nabla \eta}{\infty\!}{\!E}^2 
      +
      \big\| \ze' \big\|_{\raisebox{-0ex}{$\scriptstyle L^{\raisebox{.1ex}{$\scriptscriptstyle \!\infty$}} (I)$}}
      +
      h^{\om}(\phi)
    \Big)\; \int_{I}\; \Norm{v_t^{2 \al}}{1 , B, \mu^{\om} }\, \md t.
  \end{align}
\end{corro}
\begin{proof}
  By multiplying both sides of \eqref{eq:DF:ode} with $\ze(t)$ and integrating the resulting inequality over $[s_1, s]$ for any $s \in I$, we get
  \begin{align}\label{eq:DF:int}
    &\ze(s)\, \Norm{(\eta\, v_s^{\al})^2}{1, B, \mu^{\om}}
    \,+\,
    \int_{s_1}^s\! \ze(t)\; \frac{\cE_{\eta^2}^{\om}(v_t^{\al})}{|B|}\, \md t
    \nonumber\\[.5ex]
    &\mspace{36mu}\leq\;
    C_1 \al^2\,
    \Big( 
      \norm{\nabla \eta}{\infty\!}{\!E}^2 
      +
      \big\|
        \ze'
      \big\|_{\raisebox{-0ex}{$\scriptstyle L^{\raisebox{.1ex}{$\scriptscriptstyle \!\infty$}} (I)$}}
      +
      h^{\om}(\phi)
    \Big)\; \int_{I}\; \Norm{v_t^{2 \al}}{1, B, \mu^{\om}}\, \md t.
  \end{align}
  Thus, by neglecting the first term on the left-hand side of \eqref{eq:DF:int}, \eqref{eq:DF} is immediate, whereas \eqref{eq:max} follows once we neglect the second term on the left-hand side of \eqref{eq:DF:int}.
\end{proof}
For any $x_0 \in V$, $\th \in (0,1)$ and $n \geq 1$, we write $Q(x_0, n) \equiv [0,  \th n^2] \times B(x_0, n)$.  Further, we consider a family of intervals $\{I_{\si} : \si \in [0, 1]\}$, i.e.\
\begin{align*}
  I_{\si}
  \;\ldef\;
  \Big[
     \big( 1 - \si \big) s',(1-\si) s''+\si \theta n^2
  \Big]
\end{align*}
interpolating between the intervals $[0, \th n^2]$ and $[s', s'']$, where $s' = \ve \theta n^2$ and  $s'' =(1- \ve) \th n^2$ for any fixed $\ve \in (0,1/4)$.  Moreover, set $ Q(x_0, \si n) \ldef I_{\si} \times B(x_0, \si n)$.  (This corresponds to the notation in \cite[Section 4]{ADS13a} with the choice $t_0 = 0$, but with an additional parameter $\th$.) In addition, for any sets $I$ and $B$ as in Lemma~\ref{lem:moser_pre}  let us introduce a $L^p$-norm on functions $u\!:\bbR \times V \to \bbR$ by
\begin{align*}
  \Norm{u}{p,I \times B, \mu^{\om}}
  \;\ldef\;
  \bigg(
    \frac{1}{|I|}\; \int_{I}\; \Norm{u_t}{p, B, \mu^{\om}}^p\; \md t
  \bigg)^{\!\!1/p},
\end{align*}
where $u_t=u(t,.)$, $t\in \bbR$. 
From now on we consider a function $\phi$ of the form
\begin{align} \label{def:phi_n}
  \phi(x)
  \;=\;
  \me^{\psi(x)},
  \qquad x \in V,
\end{align}
for some function $\psi$ satisfying $\psi, \psi^{-1} \in \ell^\infty(V)$ and $\nabla \psi \in \ell^\infty(E)$ to be chosen later. %Set $\la \ldef \|\nabla \psi\|_{\ell^{\infty}(E)}$.  % Then, in view of \eqref{eq:phi:property} above note that
% %
% %
% \begin{align}\label{eq:bound:phi_n}
%   A(\phi)
%   \;=\;
%   \frac{1}{2}\, \| \cosh(\nabla \psi) + 1 \|_\infty 
%   \;\leq\;
%   \cosh(\la). 
% \end{align}

The Moser iteration can now be carried out as in \cite[Proposition~4.2]{ADS13a}.
\begin{prop} \label{prop:mos_it}
  For a given $\phi $ as in \eqref{def:phi_n} let $v > 0$ be such that $\partial_t v - \cL_{\phi}^{\om} v = 0$ on $Q(x_0, n)$ for $n \geq 2 ( N_1(x_0) \vee N_2(x_0) )$.  Then, for any $p, q \in (1,\infty]$ with
  \begin{align*}
    \frac{1}{p} \,+\, \frac{1}{q} \;<\; \frac{2}{d'}
  \end{align*}
  there exists $C_2 \equiv C_2(d, p, q)$ such that for any $1/2 \leq \si' < \si \leq 1$, $\th \in (0,1)$ and $\ve \in (0, 1/4)$,
  \begin{align}\label{eq:mos_it}
    \max_{(t,x) \in Q(x_0, \si' n)} v(t,x)
    \;\leq\;
    C_2\,
    \bigg(
      \Big(
        1 + \th n^2 h^{\om}(\phi)
      \Big)\, \frac{m^{\om}(n)}{\ve \th (\si-\si')^2}
    \bigg)^{\!\ka}\,
    \Norm{v}{2, Q(x_0, \si n), \mu^{\om}},
  \end{align}
  where
  \begin{align*}
    2\ka
    \ldef
    1+\frac{(p-1)/p}{2/d'-(p^{-1}+q^{-1})},
    \mspace{24mu}
    m^{\om}(n)
    \ldef
    \big(1 \vee \Norm{\mu^{\om}}{p,B(x_0, n)}\big)
    \big(1 \vee \Norm{\nu^{\om}}{q,B(x_0, n)}\big).
  \end{align*}
\end{prop}
\begin{proof}
  We proceed as in the proof of \cite[Proposition~4.2]{ADS13a}, which is based on the Sobolev inequality in \cite[Proposition~3.5]{ADS13}. In order to lighten notation, we set $B(n) \equiv B(x_0,n)$.  Consider a sequence $\{B(\si_k n)\}_k$ of balls with radius $\si_k n$ centered at $x_0$, where
  \begin{align*}
    \si_k \;=\; \si' + 2^{-k} (\si - \si')
    \qquad \text{and} \qquad
    \tau_k \;=\; 2^{-k-1} (\si - \si'),
    \quad k = 0, 1, \ldots
  \end{align*}
  For every $k$ such that $\tau_k n \geq 1$, let $\{\eta_k\}_k$ be a sequence of cut-off functions in space with the following properties: $\supp \eta_k \subset B(\si_k n)$, $\eta_k \equiv 1$ on $B(\si_{k+1} n)$, $\eta_k \equiv 0$ on $\partial B(\si_k n)$ and $\norm{\nabla \eta_k}{\infty}{E} \leq 1/(\tau_{k} n)$.  Further, let $\{\ze_k\}_k$ be a sequence of smooth cut-off functions in time, that is $\ze_k \in C^{\infty}(\bbR)$, $\supp \ze_k \subset I_{\si_k}$, $\ze_k \equiv 1$ on $I_{\si_{k+1}}$, $\ze_k( (1-\si_k)s') = 0$ and $\| \ze_k' \|_{\raisebox{-0ex}{$\scriptstyle L^{\raisebox{.1ex}{$\scriptscriptstyle \!\infty$}} ([0,\theta n^2])$}} \leq  1 /  (\ve \tau_k \th n^2)$.  Finally set $\al_k = (1 + (\rho - p_*) / \rho )^k$ with $\rho = d'/\big(d'-2 + d'/q \big)$ from the Sobolev inequality in \cite{ADS13}.  Since $1/p + 1/q < 2/d'$ we have $\rho > p_*$ and therefore $\al_k \geq 1$ for every $k \in \bbN_0$.  Now with these choices the equations \eqref{eq:DF} and \eqref{eq:max} become
  \begin{align*}
    \int_{I_{\si_{k+1}}} 
    \frac{\cE_{\eta_k^2}^{\om}(v_t^{\al_k})}{|B(\si_k n)|} \, \md t 
    \;\leq\;
    c \, \frac{\al_k^2}{n^2}
    \left(\frac{1}{\ve \th \tau_k^2} + n^2 h^{\om}(\phi) \right)\,
    \int_{I_{\si_k}}\! \Norm{v_t^{2 \al_k}}{1,B(\si_k n), \mu^{\om}}\, \md t
  \end{align*}
  and
  \begin{align*}
    \max_{t \in I_{\si_{k+1}}}
    \Norm{(\eta\, v_t^{\al_k})^2}{1, B(\si_k n), \mu^{\om}} 
    \;\leq\;
    c\, \frac{\al_k^2}{n^2}
    \left(\frac{1}{\ve \th \tau_k^2} + n^2 h^{\om}(\phi) \right)\,
    \int_{I_{\si_k}}\! \Norm{v_t^{2 \al_k}}{1,B(\si_k n), \mu^{\om}}\, \md t.
  \end{align*}
  The claim follows now line by line from the Moser iteration in the proof of \cite[Proposition~4.2]{ADS13a} with $\ka = \frac{1}{2} \sum_{k=0}^\infty 1/\al_k$.
\end{proof}
Finally, together with the apriori estimate, we deduce the following maximal inequality for $v$. 
\begin{corro} \label{cor:max_predual}
  Suppose that the assumptions of the Proposition~\ref{prop:mos_it} are satisfied.  Then, there exists $C_3 = C_3(d', p, q, \ve)$ such that
  \begin{align*}
    \max_{(t,x) \in Q(x_0, n/2)} v(t,x)
    \;\leq\;
    \frac{C_3}{n^{d/2}}\, \bigg(\frac{m^{\om}(n)}{\ve\th}\bigg)^{\!\!\ka}\,
    \me^{2 h^{\om}(\phi) (1-\ve) \th n^2}\,
    \norm{\phi f}{2}{V, \mu^{\om}}.
  \end{align*}
\end{corro}
\begin{proof}
  Choosing $\si' = 1/2$ and $\si = 1$ we combine Proposition~\ref{prop:mos_it} with the a-priori estimate in \eqref{eq:hke:apriori} to obtain
  \begin{align*}
    \max_{(t,x) \in Q(x_0, n/2)} v(t,x)
    \;\leq\;
    c\,
    \bigg(
      \Big(
        1 + \th n^2 h^{\om}(\phi)
      \Big)\, \frac{m^{\om}(n)}{\ve \th}
    \bigg)^{\!\ka}\,
    \me^{h^{\om}(\phi) \th n^2}\, n^{-d/2}\, \norm{\phi f}{2}{V, \mu^{\om}}.
  \end{align*}
  Since for any $\ve \in (0, 1/2)$ there exists $c(\ve) < \infty$ such that for all $n \geq 1$ and $\th \in (0,1]$
  \begin{align*}
    \Big( 1 + \th n^2 h^{\om}(\phi) \Big)^{\!\ka}\,
    \me^{-(1-2\ve)h^{\om}(\phi) \th n^2}
    \;\leq\;
    c(\ve)
    \;<\;
    \infty,
  \end{align*}
  and the claim follows.
\end{proof}

\subsection{Heat kernel bounds}
We now return to the Cauchy problem \eqref{eq:cauchy_prob}.  
\begin{prop} \label{prop:est_cp}
  Suppose that Assumption~\ref{ass:lim_const} holds and let $x_0 \in V$ be fixed. Then,  for any given $x \in V$ and $t$ with $\sqrt{t} \geq N(x_0) \vee 2(N_1(x_0) \vee N_2(x_0))$ the solution $u$ of the Cauchy problem in \eqref{eq:cauchy_prob} satisfies
  \begin{align*}
    |u(t,x)|
    \;\leq\;
    C_4\,t^{-d/2}
    \sum_{y\in V}\! \bigg( \!1 + \frac{d(x_0,x)}{\sqrt{t}}\bigg)^{\!\!\ga}
    \bigg( \!1 + \frac{d(x_0,y)}{\sqrt{t}}\bigg)^{\!\!\ga}
    \me^{\psi(y)-\psi(x)+2h^{\om}(\phi) t}\, f(y)\, \mu^{\om}(y)
  \end{align*}
  with $\ga \ldef 2\ka - d/2$ and $C_4 = C_4(d, p, q, \bar{\mu}_p, \bar{\nu}_q)$.
\end{prop}
\begin{proof}
  We will mainly follow the proof of Theorem 1.1 in \cite{Zi13}. Recall that $\ve \in (0, 1/4)$ is arbitrary, so we choose now for instance $\ve \ldef 1/8$.  By Assumption~\ref{ass:lim_const} we have
  \begin{align*}
    m^{\om}(n)
    \;\leq\;
    4 \big( 1 \vee \bar \mu_p \big) \big( 1\vee \bar \nu_q \big)
    \qquad \text{for all $n \geq N(x_0)$.}
  \end{align*}
 Next, for any given $x \in V$ and $t$ as in the statement we choose $n$ and $\theta$ in such a way that $(t,x) \in Q(x_0, n/2)$ (for this purpose we need the additional parameter $\theta$).  We set 
  \begin{align*}
    n
    \;=\;
    \big\lceil 2d(x_0,x) +  \sqrt{8t/7}\, \big\rceil
    \;\geq\;
    N(x_0) \vee 2 \big( N_1(x_0) \vee N_2(x_0) \big) 
  \end{align*}
  and $\th \ldef t/\frac{7}{8} n^2$ so that $t = \frac{7}{8} \th n^2 = s''$ and $(t,x) \in Q(x_0, n/2)$.  Then, Corollary~\ref{cor:max_predual} implies that
  \begin{align*}
    \me^{\psi(x)} u(t,x)
    \;\leq\;
    c\,  t^{-\ka}\, \me^{2h^{\om}(\phi) t}\, n^{\ga}\,
    \norm{\me^{\psi} f}{2}{V, \mu^{\om}}.
  \end{align*}
  This can be rewritten as
  \begin{align}\label{eq:L-infinity-2}
    \norm{ b^{-1}(t, \cdot)\,
    P_t^{\psi}\big(\me^\psi f\big) }{\infty}{V, \mu^\om}
    \;\leq\;
    c\, r(t)\, \norm{\me^{\psi} f}{2}{V,\mu^\om},
  \end{align}
  where $P_t^{\psi} g \ldef \me^{\psi}\, P_t(\me^{-\psi} g)$ and 
  \begin{align*}
    r(t)
    \;\rdef\;
    t^{-\ka}\, \me^{2 h^{\om}(\phi) t},
    \qquad
    b(t,x)
    \;\ldef\;
    \bigg( 2\, d(x_0,x) +  \sqrt{\tfrac 8 7 t }\, \bigg)^{\!\!\ga}.
  \end{align*}
  Notice that $P_t^{-\psi}$ is the adjoint of $P_t^{\psi}$ in $\ell^2(V, \mu^{\om})$.  Since $h(\phi)$ remains unchanged if we replace $\psi$ by $-\psi$, \eqref{eq:L-infinity-2} also holds true for $\psi$ replaced by $-\psi$.  Therefore, we get by duality that for all $g \in \ell^1(V, \mu^{\om})$,
  \begin{align}\label{eq:L-2-1}
    \norm{ P_t^{\psi} \big( b^{-1}(t, \cdot) g \big)}{2}{V, \mu^\om}
    \;\leq\;
    c \, r(t)\,  \norm{g}{1}{V, \mu^{\om}}.
  \end{align}
  Since $b(t/2,x) \leq b(t,x)$,
  \begin{align*}
    \norm{b^{-1}(t,\cdot)\, \me^{\psi} P_t f }{\infty}{V,\mu^\om}
    &\;\leq\;
    \norm{b^{-1}(t/2, \cdot)\, \me^{\psi} P_{t/2}\big( P_{t/2}f \big) }
    {\infty}{V,\mu^\om}
    \\[.5ex]
    &\overset{\eqref{eq:L-infinity-2}}{\;\leq\;}
    c\, r(t/2)\, \norm{ \me^{\psi}\, P_{t/2}f }{2}{V,\mu^\om}
    \\[.5ex]
    &\overset{\eqref{eq:L-2-1}}{\;\leq\;}
    c^2\, r(t/2)^2\, \norm{ \me^{\psi}\, b(t/2, \cdot) f}{1}{V,\mu^\om}.
  \end{align*}
  Hence,
  \begin{align*}
    |u(t,x)|
    \;\leq\;
    \frac{c}{t^{2\ka}}\, \me^{2 h^{\om}(\phi) t -\psi(x)}\,
    \Big( d(x_0,x) +  \sqrt{t}\Big)^{\!\ga}
    \sum_{y\in V} \Big( d(x_0,y) +  \sqrt{t} \Big)^{\!\ga}\,
    \me^{\psi(y)} f(y)\, \mu^\om(y),
  \end{align*}
  which is the claim since $\ga = 2 \ka - d/2$.
\end{proof}
\begin{proof}[Proof of Theorem~\ref{thm:hke}]
  We apply Proposition~\ref{prop:est_cp} on  the heat kernel $q^{\om}(t, x, y)$, that is $f = \indicator_{\{y\}}/\mu^{\om}(y)$, which yields
  \begin{align*}
    q^{\om}(t, x, y)
    \;\leq\;
    C_4\, t^{-d/2}\,
    \bigg( 1 + \frac{d(x_0,x)}{\sqrt t} \bigg)^{\!\!\ga}
    \bigg( 1 + \frac{d(x_0,y)}{\sqrt t} \bigg)^{\!\!\ga}\,
    \me^{\psi(y)-\psi(x) + 2h^{\om}(\phi) t}. 
  \end{align*}
  We now optimize over $\phi = \me^{\psi}$.  Let
  \begin{align*}
    \psi(u)
    \;\ldef\;
    - \la\, \min \big\{ d(x,u), d(x,y) \big\}.
  \end{align*}
  Since $h^{\om}(\phi) \leq 2 (\cosh(\la) - 1)$ this gives
  \begin{align*}
    \exp\!\big(\psi(y) - \psi(x) + 2h(\phi) t\big)
    &\;\leq\;
    \exp\!\Big( -\la\, d(x,y) + 2\,t\, \big( \me^\la + \me^{-\la} - 2 \big)\Big)
    \\[.5ex]
    &\;=\;
    \exp\!\bigg(
      d(x,y) \Big( -\la + \frac{2t}{d(x,y)}\, \big( e^\la+e^{-\la}-2 \big)\Big)
    \bigg).
  \end{align*}
  So if
  \begin{align*}
    F(s)
    \;=\;
    \inf_{\la >0} \Big( -\la + (2s)^{-1} \big( e^\la+e^{-\la}-2 \big) \Big),
  \end{align*}
  then
  \begin{align} \label{eq:q_estF}
    q^\om(t,x, y)
    \;\leq\;
    c\, t^{-d/2}\, \bigg( 1 + \frac{d(x_0,x)}{\sqrt t} \bigg)^{\!\!\ga}
    \bigg( 1 + \frac{d(x_0,y)}{\sqrt t} \bigg)^{\!\!\ga}\,
    \exp\!\bigg( d(x,y)\, F\bigg(\frac{d(x,y)}{4t}\bigg) \bigg).
  \end{align}
  We have 
  \begin{align*}
    F(s)
    \;=\;
    s^{-1} \big( (1 + s^s)^{1/2} - 1 \big)
    \,-\, \log\big( s+(1+s^2)^{1/2}\big)
  \end{align*}
  and also $F(s)\leq -s/2(1-s^2/10)$ for $s>0$ (see \cite{BD10} and \cite[page 70]{Da93}).  Hence, if $s\leq 3$, then $F(s)\leq -s/20$ while if $s\geq \me$, then
  \begin{align*}
    F(s)
    \;\leq\;
    1 - \log(2s)
    \;=\;
    -\log(2s/e).
  \end{align*}
  Now, by choosing $x = x_0$ and substituting in \eqref{eq:q_estF} we obtain that there exist suitable constants $c_1, \ldots, c_4$ such that if $d(x_0,y)\leq c_1 t$ then
  \begin{align*}
    q^\om(t,x_0, y)
    \;\leq\;
    c_2\, t^{-d/2}\,
    \bigg( 1 + \frac{d(x_0,y)}{\sqrt t} \bigg)^{\!\!\ga}\,
    \exp\!\bigg( \!\!-\! 2c_3\, \frac{d(x_0,y)^2}{t}\bigg)
  \end{align*}
  and if $d(x_0,y) \geq c_1 t$ then
  \begin{align*}
    q^\om(t, x_0, y)
    \;\leq\;
    c_2\, t^{-d/2}\,
    \bigg( 1 + \frac{d(x_0,y)}{\sqrt t}\bigg)^{\!\!\ga}\,
    \exp\!\bigg(%
      \!-\!2 c_4\, d(x_0,y)\, \bigg( 1 \vee \log\frac{d(x_0,y)}{t} \bigg)
    \bigg).
  \end{align*}
  Finally, since
  \begin{align*}
    \bigg( 1 + \frac{d(x_0,y)}{\sqrt{t}} \bigg)^{\!\!\ga}\,
    \exp\!\bigg( \!\!-\!c_3\, \frac{d(x_0,y)^2}{t} \bigg)
    \;\leq\;
    \sup_{z \geq 0}\, (1+z)^{\ga}\, \me^{-c_3 z^2}
    \;\leq\;
    c,
  \end{align*}
  and
  \begin{align*}
    &\bigg( 1 + \frac{d(x_0,y)}{\sqrt{t}} \bigg)^{\!\!\ga}\,
    \exp\!\bigg(%
      \!\!-\! c_4\, d(x_0,y)\, \bigg( 1 \vee \log\frac{d(x_0,y)}{t} \bigg)
    \bigg)
    \\[1ex]
    &\mspace{36mu}\leq\;
    \big( 1 + d(x_0,y) \big)^{\ga}\,
    \exp\!\big( \!-\! c_4 d(x_0,y) \big)
    \;\leq\;
    \sup_{z \geq 0}\, (1+z)^{\ga}\, \me^{-c_4 z}
    \;\leq\;
    c,
  \end{align*}
  after adapting the constant $c_2$ we obtain the result.
\end{proof}

\section{Gaussian upper bounds for the VSRW} \label{sec:GBvsrw}
Theorem~\ref{thm:hke_VSRW} can be proven essentially along the lines of the proof of Theorem~\ref{thm:hke} above. For the reader's convenience we explain the main adjustments in this section. 

For any non-empty, finite $B \subset V$, any interval $I\subset \bbR$ and $p,q \in [1, \infty)$, we introduce a space-averaged and a space-time averaged norm on functions $f\!: B \to \bbR$ and $u\!: I \times B \to \bbR$, respectively,  by
\begin{align*}
  \Norm{f}{p,B}
  \;\ldef\;
  \bigg(
    \frac{1}{|B|}\, \sum_{x \in B}  |f(x)|^p
  \bigg)^{\!\!1/p}
  \;\quad \text{and} \qquad
  \Norm{u}{p, q, I \times B}
  \;\ldef\;
  \bigg(
    \frac{1}{|I|}\; \int_{I}\; \Norm{u_t}{p, B}^q\; \md t
  \bigg)^{\!\!1/q}\mspace{-18mu},
\end{align*}
where $u_t = u(t,.)$, $t\in \bbR$. 
Consider now the Cauchy problem for the operator $\cL^{\om}_{\mathrm{V}}$,
\begin{align}\label{eq:cauchy_prob_V}
   \left\{
    \begin{array}{rcl}
      \partial_t u - \cL^{\om}_{\mathrm{V}} u
      &\mspace{-5mu}=\mspace{-5mu}& 0,
      \\[1ex]
      u(t=0, \,\cdot\,)
      &\mspace{-5mu}=\mspace{-5mu}& f.
    \end{array}
  \right.
\end{align}
With a slight abuse of notation, for any positive function $\phi$ on $V$ such that $\phi, \phi^{-1} \in \ell^{\infty}(V)$ let  $(\cL_{\phi}^{\om}\, g)(x) \ldef \phi(x) (\cL^{\om}_{\mathrm{V}} \phi^{-1} g)(x)$ acting on bounded functions $g\!: V \to \bbR$.
The a-priori estimate in \eqref{eq:hke:apriori} now reads as follows.
\begin{lemma}\label{lemma:hke:apriori_V}
  Suppose that $f \in \ell^2(V)$ and $u$ solves the Cauchy problem \eqref{eq:cauchy_prob_V}.  Set $v(t,x) \ldef \phi(x) u(t,x)$ for a positive function $\phi$ on $V$ such that $\phi, \phi^{-1} \in \ell^{\infty}(V)$.  Then
  \begin{align}\label{eq:hke:apriori_V}
    \norm{v(t, \,\cdot\,)}{2}{V}
    \;\leq\;
    \me^{\tilde{h}^{\om}(\phi) t}\, \norm{\phi f}{2}{V},
  \end{align}
  where
  \begin{align*}
   \tilde{h}^{\om}(\phi)
    \;\ldef\;
    \max_{x \in V}\,
    \sum_{y \sim x}\, \big| \md\Ga^{\om}(\phi, \phi^{-1})(\{x,y\}) \big|.
    % \max_{e \in E} \big(1\vee \om(e)\big)
    % \bigg(
    %   \frac{\phi(e^+)}{\phi(e^-)} \,+\, \frac{\phi(e^-)}{\phi(e^+)} \,-\, 2
    % \bigg).
  \end{align*}
\end{lemma}
\begin{proof}
  This can be proven similarly as Lemma~\ref{lem:apriori}.
\end{proof}
\begin{lemma} \label{lem:prep_moserV}
 Let $Q$, $\eta$, $\ze$ and $\phi$ be as in Lemma~\ref{lem:moser_pre} and Corollary~\ref{cor:prep_moser}, and let $v_t\geq 0$ satisfy $\partial_t v-\cL^{\om}_{\phi} v \leq 0$.  Then, there exists $C_5 < \infty$ such that for all $\al \geq 1$ and $p \geq 1$,
  \begin{align}\label{eq:DFV}
    &\max_{t \in I} \Big( \ze(t)\; \Norm{(\eta\, v_t^{\al})^2}{1, B} \Big)
    \,+\, \int_{I} \ze(t)\; \frac{\cE_{\eta^2}^{\om}(v_t^{\al})}{|B|} \, \md t
    \nonumber\\[.5ex]
    &\mspace{36mu}
    \;\leq\;
    C_6\, \al^2\, \Big( 1 \vee \Norm{\mu^{\om}}{p, B} \Big)
    \Big( 
      \norm{\nabla \eta}{\infty\!}{\!E}^2
      +
      \big\| \ze' \big\|_{\raisebox{-0ex}{$\scriptstyle L^{\raisebox{.1ex}{$\scriptscriptstyle \!\infty$}} (I)$}}
      + \tilde{h}^{\om}(\phi)
    \Big)\, \int_{I} \Norm{v_t^{2 \al}}{p_* , B}\, \md t,
    \nonumber\\
  \end{align}
  where $p_*:=p/(p-1)$.
\end{lemma}
\begin{proof}
  Since $\partial_t v - \cL^{\om}_{\phi} v \leq 0$ on $Q$, we have for every $t \in I$ and $\al \geq 1$
  \begin{align*}
    \frac{1}{2 \al}\,
    \frac{\md}{\md t}\, \norm{(\eta v_t^{\al})}{2}{V}^2
    \;\leq\;
    \scpr{\eta^2 v_t^{2 \al-1}}{\cL_{\phi}^{\om} v_t}{V}
    \;=\;
    - \scpr{\nabla (\eta^2 \phi\, v_t^{2 \al-1})}
    {\om \nabla (\phi^{-1} v_t)}{E}.
  \end{align*}
  By following line by line the arguments in the proof of Lemma~\ref{lem:moser_pre} and by applying H\"older's inequality on the resulting term $\|v_t^{2\al}\|_{1,B,\mu^\om}$ we obtain
  \begin{align*}
    \frac{\md}{\md t}\, \Norm{(\eta v_t^{\al})}{2, B}^2
    \,+\,
    \frac{\cE_{\eta^2}^{\om}(v_t^{\al})}{|B|}
    \leq
    c \al^2\, \Norm{\mu^\om}{p,B}\,
    \Big( \norm{\nabla \eta}{\infty}{E}^2 + \tilde{h}^{\om}(\phi)\Big)\,
    \Norm{v_t^{2\al}}{p_*, B}.
  \end{align*}
  By multiplying both sides  with $\ze(t)$ and integrating over $[s_1, s]$ for any $s \in I$, we get
  \begin{align*}
    \ze(s)\, \Norm{(\eta\, v_s^{\al})^2}{1, B}
    &\,+\,
    \int_{s_1}^s\! \ze(t)\; \frac{\cE_{\eta^2}^{\om}(v_t^{\al})}{|B|}\, \md t
    \,-\,
    \big\|
        \ze'
      \big\|_{\raisebox{-0ex}{$\scriptstyle L^{\raisebox{.1ex}{$\scriptscriptstyle \!\infty$}} (I)$}}
    \, \int_{I}\; \Norm{v_t^{2 \al}}{1, B}\, \md t
    \nonumber\\[1ex]
    &\mspace{36mu}
    \leq\; 
    c \al^2\, \Norm{\mu^\om}{p,B}\,
    \Big(\norm{\nabla \eta}{\infty\!}{\!E}^2 + \tilde{h}^{\om}(\phi)\Big)\,
    \int_{I}\; \Norm{v_t^{2 \al}}{p_*, B}\, \md t.
  \end{align*}
  Since by Jensen's inequality $\|v_t^{2 \al}\|_{1, B} \leq \|v_t^{2 \al}\|_{p_*, B}$ for every $t$, the claim follows.
\end{proof}
%
% \textcolor{red}{***** Revise the remaining part *****}
% From now on we will again consider a function $\phi$ of the form $\me^\psi$.  Set
% %
% \begin{align*}
%   \la_1
%   \;=\;
%   \la_1(\om)
%   \;\ldef\;
%   \max_{e \in E} \big( 1\vee \om(e) \big)^{1/2} \big| \nabla \psi(e) \big|. 
% \end{align*}
% %
% Then, since $\cosh(x)\geq 1+\frac  12  x^2$ for all $x\in \bbR$ we have 
% %
% \begin{align} \label{eq:cond_phi}
%   \tilde{h}_{\om}(\phi) 
%   \;=\;
%   2\, \max_{e\in E}
%   \Big\{ \big(1\vee \om(e)\big)  \big( \cosh(\nabla \psi) -1 \big) \Big\}
%   \;\geq\;
%   \la_1^2.
% \end{align}  
% %
% Similarly as in the previous section note that by monotonicity
% %
% \begin{align} \label{eq:bound:phiV}  
%   % h(\phi^2)
%   % \;\leq \;
%   % 4 \cosh^2(\la_1), \qquad 
%   A(\phi)
%   \;\leq\;
%   \cosh(\la_1). 
% \end{align}
% %
% For any $x_0 \in V$, $\th \in (0,1)$ and $n \geq 1$ let  $\{ I_\si: \si\in [0,1]\}$  still be defined  as below Corollary~\ref{cor:prep_moser}. Further, we set $\tilde Q(x_0, \si n) \ldef I_\si \times \tilde B(x_0, \si n)$.
By Moser iteration we obtain the following maximal inequality for perturbed solutions of the Cauchy problem \eqref{eq:cauchy_prob_V}.
\begin{prop}
  Let $v > 0$ be such that $\partial_t v - \cL_{\phi}^{\om} v = 0$ on $Q(x_0, n)$ for any $n \geq 2 ( N_1(x_0) \vee N_2(x_0) )$.  Then, for any $\th \in (0,1)$, $\ve \in (0, 1/4)$ and $p, q \in (1,\infty]$ with
  \begin{align*}
    \frac{1}{p} \,+\, \frac{1}{q} \;<\; \frac{2}{d'},
  \end{align*}
  there exists $C_7 \equiv C_7(d', p, q, \ve)$ and $\ka = \ka(d', p, q)$ such that
  \begin{align} \label{eq:max_ineqV}
    \max_{(t,x) \in Q(x_0, n/2)} v(t,x)
    \;\leq\;
    \frac{C_7}{n^{d/2}}\,
    \bigg(
      \frac{m^{\om}(n)}{\ve \th}
    \bigg)^{\!\!\ka}\,
    \me^{2\tilde{h}^{\om}(\phi) (1-\ve) \th n^2}\,
    \norm{\phi f}{2}{V}.
  \end{align}
\end{prop}
\begin{proof}
  As in the proof of Proposition~\ref{prop:mos_it} above we will basically follow the arguments of \cite[Proposition 4.2]{ADS13a}.  Set $\al = 1/p_* + 1/\rho_*$ and $\al_k = \al^k$, where  $\rho_*$ is the H\"older-conjugate of $\rho = d'/(d'-2+d'/q)$.  Note that for any $p, q \in (1, \infty]$ for which $1/p + 1/q < 2/d'$ is satisfied, $\al = 1/p_* + 1/\rho_* > 1$ and therefore $\al_k \geq 1$ for every $k \in \bbN_0$.  Moreover, for some $1/2 \leq \si' < \si \leq 1$, let $\si_k$ and $\tau_k$ be defined as in the proof of Proposition~\ref{prop:mos_it}.

  Consider first the case $\tau_k n \geq 1$ (i.e.\ $B(x_0, \si_{k+1} n) \subsetneq B(x_0, \si_k n)$), and choose the same cut-off function $\{\eta_k\}_k$ and $\{\ze_k\}_k$ as before.  Now, by H\"{o}lder's inequality we have that
  \begin{align}\label{eq:Moser_VSRW_split}
    &\Norm{v^{2 \al_k}}{\al p_*, \al, Q(x_0,\si_{k+1} n)}
    \nonumber\\
    % &\mspace{36mu}\leq\;
    % \bigg(
    %   \frac{1}{|I_{\si_{k+1}}|}\;
    %   \int_{I_{\si_{k+1}}}\mspace{-18mu}
    %     \Norm{v_t^{2 \al_k}}{1, B(\si_{k+1}n)}^{\al-1}\,
    %     \Norm{v_t^{2 \al_k}}{\rho, B(\si_{k+1} n)}\,
    %   \md t
    % \bigg)^{\!\!1/\al}
    % \nonumber\\[0.5ex]
    &\mspace{36mu}\leq\;
    \bigg(
      \max_{t \in I_{\si_{k+1}}} \Norm{v_t^{2 \al_k}}{1, B(x_0, \si_{k+1}n)}
    \bigg)^{\!\!1-1/\al}\,
    % \bigg(
    %   \frac{1}{|I_{\si_{k+1}}|}\;
    %   \int_{I_{\si_{k+1}}}\mspace{-18mu}
    %   \Norm{v_t^{2\al_k}}{\rho, B(x_0, \si_{k+1} n)}\, \md t
    % \bigg)^{\!\!1/\al}\mspace{-15mu}.
    \Norm{v^{2 \al_k}}{\rho, 1, Q(x_0, \si_{k+1} n)}^{1/\al}.  
  \end{align}
  By the Sobolev inequality in \cite[Proposition 3.5]{ADS13} the integrand can be estimated from above by
  \begin{align*}
    &\Norm{(\eta_k v_t^{\al_k})^2}{\rho, B(x_0, \si_k n)}  
    \\
    &\mspace{36mu}\leq\;
    C_{\mathrm{S}}\, n^2\, \Norm{\nu^{\om}}{q, B(x_0, \si_k n)}
    \Bigg(
      \frac{\cE_{\eta_k^2}^{\om}\big(v_t^{\al_k}\big)}{|B(x_0, \si_k n)|}
      \,+\,
      \frac{1}{(\tau_k n)^2}\, %\Norm{\mu^{\om}}{p, B(x_0, \si_k n)}
      \Norm{v_t^{2 \al_k}}{1, B(x_0, \si_k n)}\!
    \Bigg).
  \end{align*}
  Then, an application of the energy estimates as derived in Lemma~\ref{lem:prep_moserV} yields
  \begin{align*}
    &\max_{t \in I_{\si_{k+1}}}
    \Norm{(\eta_k v_t^{\al})^2}{1, B(x_0, \si_k n)}
    \,+\,
    \int_{I_{\si_{k+1}}} 
    \frac{\cE_{\eta_k^2}^{\om}(v_t^{\al_k})}{|B(x_0, \si_k n)|}\, \md t
    \\[.5ex]
    &\mspace{32mu}\overset{\eqref{eq:DFV}}{\;\leq\;}
    c\,
    \al_k^2\, \Big(1 \vee \Norm{\mu^{\om}}{p, B(x_0, n)}\Big)\,
    \bigg(
      \frac{1}{\ve \th \tau_k^2} + n^2 \tilde{h}^{\om}(\phi)
    \bigg)\,
    \Norm{v^{2 \al_k}}{p_*, 1, Q(x_0, \si_k n)}.
    %\int_{I_{\si_k}}\! \Norm{v_t^{2 \al_k}}{p_*,\tilde B(\si_k n)}\, \md t,
  \end{align*}
  % % 
  % and
  % %
  % \begin{align*}
  %   \int_{I_{\si_{k+1}}} 
  %   \frac{\cE_{\eta_k^2}^{\om}(v_t^{\al_k})}{|B(\si_k n)|}\, \md t 
  %   \overset{\eqref{eq:DFV}}{\;\leq\;}
  %   c\,
  %   \al_k^2\, \Norm{\mu^{\om}}{p, B(n)}\,
  %   \bigg(
  %     \frac{1}{\ve \th \tau_k^2} + n^2 \tilde{h}^{\om}(\phi)
  %   \bigg)\,
  %   \Norm{v^{2 \al_k}}{p_*, 1, Q(x_0, \si_k n)}.
  % \end{align*}
  %
  Combining the estimates above yields
  \begin{align}
    &\Norm{v}{2 \al_{k+1} p_*, 2 \al_{k+1}, Q(x_0, \si_{k+1} n)}
    \nonumber\\[.5ex]  
    &\mspace{36mu}\leq\;
    c\, 2^{k/\al_k}\,
    \bigg(
      \Big( 1 + \th n^2 \tilde{h}^{\om}(\phi) \Big)
      \frac{m^{\om}(n)}{\ve \th (\si - \si')^2}
    \bigg)^{\!\!1/2 \al_k}
    \Norm{v}{2 \al_k p_*, 2 \al_k, Q(x_0, \si_k n)}.
  \end{align}
  Next we consider the case $\tau_k n < 1$. We shall estimate both terms on the right-hand side of \eqref{eq:Moser_VSRW_split}.  Note that for any $t \in I_{\si_{k+1}}$,
  \begin{align*}
    \Norm{v_t^{2\al_k}}{\rho, B(x_0, \si_k n)}
    &\;\leq\;
    \Big(\max_{x \in B(x_0, \si_{k} n)} v_t(x)^{2\al_k}\Big)^{1-p_*/\rho}\,
    \Norm{v_t^{2\al_k}}{p_*, B(x_0,\si_k n)}^{p_*/\rho}
    \\
    &\;\leq\;
    \Big(   
      |B(x_0, \si_k n)|^{1/p_*}\, \Norm{v_t^{2\al_k}}{p_*, B(x_0, \si_k n)}
    \Big)^{1-p_*/\rho}\,
    \Norm{v_t^{2\al_k}}{p_*, B(x_0, \si_k n)}^{p_*/\rho}
    \\[1ex]
    &\;\leq\;
    |B(x_0, \si_k n)|^{1/p_* - 1/\rho}\,
    \Norm{v_t^{2\al_k}}{p_*, B(x_0, \si_k n)}.
  \end{align*}
  Since $d (1/p_* - 1/\rho) \leq 2$ and $n < 1/\tau_k$ by assumption, we find
  \begin{align*}
    \Norm{v^{2\al_k}}{\rho,1, Q(x_0, \si_{k+1} n)}
    \;\leq\;
    c\, \frac{2^{2k}}{(\si - \si')^2}\,
    \Norm{v^{2\al_k}}{p_*, 1, Q(x_0, \si_k n)}.
  \end{align*}
  Hence, it remains to estimate the first term on the right-hand side of \eqref{eq:Moser_VSRW_split}.  Since $v > 0$ is caloric on $Q(x_0, n)$, it holds for any $t \in I_{\si}$ and $x \in B(x_0, \si n)$ that $\partial_t v_t(x) = (\cL_{\phi}^{\om} v){x} \geq - \mu^{\om}(x) v_t(x)$.  In particular, a summation over all $x \in B(x_0, \si_k n)$ yields
  \begin{align*}
    \frac{1}{2\al_k} \partial_t \Norm{v_t^{2\al_k}}{1, B(x_0, \si_{k} n)}
    \;\geq\;
    - \Norm{v_t^{2\al_k}}{1, B(x_, \si_k n), \mu^{\om}}.
  \end{align*}
  Let now $\xi_k \in C^{\infty}(\bbR)$ be a cut-off function in time such that $\supp \xi_k \subset I_{k}$, $\xi_k \equiv 1$ on $I_{k+1}$, $\xi_k( t_k ) = 0$ and $\big\| \xi_k' \big\|_{\raisebox{-0ex}{$\scriptstyle L^{\raisebox{.1ex}{$\scriptscriptstyle \!\infty$}} ([0,\th n^2])$}} \leq  1 / (\ve \tau_k \th n^2)$, where we write in short $t_k\ldef (1-\si_k)s''+\si_k \th n^2$ for the right endpoint of $I_k$.  Hence,
  \begin{align*}
    \partial_t \Big( \xi_k(t) \, \Norm{v_t^{2 \al_k}}{1, B(x_0, \si_k n} \Big)
    \mspace{2mu}\geq\;
    \xi_k'(t)\, \Norm{v_t^{2 \al_k}}{1, B(x_0, \si_k n)}
    -  2\al_k\, \xi_k(t)\,
    \Norm{v_t^{2 \al_k}}{1, B(x_0, \si_k n), \mu^{\om}}.
  \end{align*}
  An integration over the interval $[t_*, t_k]$, where $t_* \in I_{k+1}$ is chosen in such a way that
  \begin{align*}
    \Norm{v_{t_*}^{2\al_k}}{1, B(x_0, \si_k n)}
    \;=\;
    \max_{t \in I_{\si_{k+1}}} \Norm{v_{t}^{2\al_k}}{1, B(x_0, \si_k n)},
  \end{align*}
  yields
  \begin{align*}
    &\max_{t\in I_{k+1}} \Norm{v_t^{2 \al_k}}{1, B(x_0, \si_k n)}
    \\[.5ex]  
    &\mspace{36mu}\leq\;
    \int_{t_*}^{t_k}
      \Big(
        2\al_k\, \xi_k(t)\, \Norm{v_t^{2 \al_k}}{1, B(x_0, \si_k n), \mu^{\om}}
        \,-\,
        \xi_k'(t)\,
        \Norm{v_t^{2 \al_k}}{1, B(x_0, \si_k n)}
      \Big)\,
    \md t.
  \end{align*}
  By using Jensen's and H\"older's inequality, we finally find that
  \begin{align*}
    \Norm{v^{2\al_k}}{1, \infty, Q(x_0, \si_{k+1} n)}
    \;\leq\;
    c \al_k 2^k\, \frac{1 \vee \Norm{\mu^{\om}}{p, B(x_0, n)}}{\ve (\si-\si')}\,
    \Norm{v^{2\al_k}}{p_*, 1, Q(x_0, \si_k n)}
  \end{align*}
  Combining the estimates above yields
  \begin{align*}
    &\Norm{v}{2\al_{k+1} p_*, 2 \al_{k+1}, Q(x_0, \si_{k+1} n)}
    \\[.5ex]  
    &\mspace{36mu}\leq\;
    c\, 2^{k/\al_k}\,
    \bigg(
      \frac{m^{\om}(n)}{\ve \th (\si-\si')^2}
    \bigg)^{\!\!1/(2\al_k)}\,
    \Norm{v}{2\al_k p_*, 2\al_k, Q(x_0, \si_k n) }.
  \end{align*}
  By following line by line the proofs of \cite[Proposition~4.2 and Corollary~4.3]{ADS13a} we get
  \begin{align*}
    \max_{(t,x) \in Q(x_0, \si' n)} v(t,x)
    \;\leq\;
    c \,
    \bigg(
      \Big(1 + \th n^2 \tilde{h}^{\om}(\phi)\Big)
      \frac{m^{\om}(n)}{\ve \th (\si - \si')^2}
    \bigg)^{\!\!\ka}\,
    \Norm{v}{2, 2, Q(x_0, \si n)},
  \end{align*}
  for some $\ka < \infty$.  Now we choose $\si' = 1/2$ and $\si = 1$.  Then, in view of Lemma~\ref{lemma:hke:apriori_V} we obtain
  \begin{align*}
    \max_{(t,x) \in Q(x_0, n/2)} v(t,x)
    \overset{\eqref{eq:hke:apriori_V}}{\;\leq\;}
    c\,
    \bigg(
      \Big( 1 + \th n^2 \tilde{h}^{\om}(\phi) \Big)\, \frac{m^{\om}(n)}{\ve \th}
    \bigg)^{\!\!\ka}\,
    \me^{\tilde{h}^{\om}(\phi) \th n^2}\, n^{-d/2}\, \norm{\phi f}{2}{V}.
  \end{align*}
  Since for any $\ve \in (0, 1/2)$ there exists $c(\ve) < \infty$ such that for all $n \geq 1$ and $\th \in (0, 1]$
  \begin{align*}
    \Big(1 + \th n^2 \tilde{h}^{\om}(\phi) \Big)^{\!\ka}\,
    \me^{-(1-2\ve) \tilde{h}^{\om}(\phi) \th n^2}
    \;\leq\;
    c(\ve)
    \;<\;
    \infty,
  \end{align*}
  the claim follows.
\end{proof}
\begin{prop}  \label{prop:est_cpV}
  Suppose that Assumption~\ref{ass:lim_const} holds and let $x_0 \in V$ be fixed. Then,  for any given $x \in V$ and $t$ with $\sqrt{t} \geq N(x_0) \vee 2(N_1(x_0) \vee N_2(x_0))$ the solution $u$ of the Cauchy problem in \eqref{eq:cauchy_prob_V} satisfies
  \begin{align*}
    |u(t,x)|
    \;\leq\;
    C_8\,t^{-d/2}\,
    \sum_{y\in V}\! \bigg(1 + \frac{d(x_0, x)}{\sqrt{t}}\bigg)^{\!\!\ga}
    \bigg( 1 + \frac{d(x_0, y)}{\sqrt{t}}\bigg)^{\!\!\ga}
    \me^{\psi(y)-\psi(x)+2\tilde{h}^{\om}(\phi) t}\, f(y)
 \end{align*}
 with $\ga \ldef 2\ka - d/2$ and $C_8= C_8(d, p, q, \bar{\mu}_p, \bar{\nu}_q)$.
\end{prop}
\begin{proof}
  Given \eqref{eq:max_ineqV} this follows by the same arguments as in Proposition~\ref{prop:est_cp}. %Note that one needs to work with the chemical distance $d^{\om}$ instead of the graph distance.
\end{proof}
\begin{proof}[Proof of Theorem~\ref{thm:hke_VSRW}]
  We apply Proposition~\ref{prop:est_cpV} on  the heat kernel $p^\om(t,x,y)$, that is $f = \indicator_{\{y\}}$, which yields
  \begin{align*}
    p^{\om}(t, x, y)
    \;\leq\;
    C_8\, t^{-d/2}\,
    \bigg( 1 + \frac{d(x_0,x)}{\sqrt t} \bigg)^{\!\!\ga}
    \bigg( 1 + \frac{d(x_0,y)}{\sqrt t} \bigg)^{\!\!\ga}\,
    \me^{\psi(y)-\psi(x) + 2\tilde h_\om(\phi) t}. 
  \end{align*}
  Now we again optimize over $\phi = \me^{\psi}$, where
  \begin{align*}
    \psi(u)
    \;\ldef\;
    - \la \, \min \big\{ d^{\om}(x,u), d^{\om}(x,y) \big\}, \qquad \la > 0.
  \end{align*}
  Recall the definition of $d^{\om}(x,y)$. Since $|\nabla \psi(e)| \leq \la \big(1 \wedge 1/\om(e)\big)^{1/2}$ for any $e \in E$ and $a \big(\cosh(x)-1\big)\leq \cosh(\sqrt{a} x)$ for all $x\in \bbR$ for any $a \geq 1$, we get
  \begin{align*}
    \tilde{h}^{\om}(\phi)
    \;=\;
    \max_{x \in V} 2 \sum_{y \sim x} \om(x,y)\,
    \big( \cosh(\nabla \psi(\{x,y\})) - 1 \big)
    \;\leq\;
    2 C_{\text{deg}}\, \big(\cosh(\la) - 1 \big).
  \end{align*}
  Hence,
  \begin{align*}
    \exp\big(
      \psi(y) - \psi(x) + 2 \tilde{h}^{\om}(\phi) t
    \big)
    \;\leq
    \exp\!%
    \bigg(
      d^{\om}(x,y)
      \Big(
        \!-\la + \frac{2tC_{\text{deg}}}{d^{\om}(x,y)}
        \big(\me^{\la} + \me^{-\la} - 2\big)
      \Big)
    \bigg).
  \end{align*}
  The claim follows now by the same arguments as in the proof of Theorem~\ref{thm:hke} above.  
\end{proof}

\appendix
\section{Proof of Lemma~\ref{lemma:chemical:distance}}
\label{sec:chemical:distance}
\begin{proof}
  Set $u_1(k) = \big(\frac{k-1}{2}/(\ln \frac{k-1}{2})^{\de} \big)^{1/\al}$ and $u_2(k) = \big(\frac{k-1}{2} (\ln \frac{k-1}{2})^{1+\de}\big)^{1/\al}\!\!$.  Obviously, the sequences $\{u_1(k) : k \geq 7 \}$ and $\{u_2(k) : k \geq 3\}$ are non-decreasing and we have $k\prob[Z_1 > u_1(k)] \to \infty$ as $k$ tends to $\infty$.  Moreover, an elementary computation shows that
  \begin{align*}
    \sum_{k=1}^{\infty}\, \prob\!\big[Z_1 > u_2(k)\big]
    \;\leq\;
    2k_0 + 2\, \sum_{k=1}^{\infty} \frac{1}{k (\ln k)^{1+\de}}
    \;<\;
    \infty
  \end{align*}
  and
  \begin{align*}
    \sum_{k=1}^{\infty}\, \prob\!\big[Z_1 > u_1(k)\big]
    \exp\Big(\!-\! k \prob\!\big[Z_1 > u_1(k)\big]\Big)
    \;\leq\;
    2k_0 + 2\sum_{k=1}^{\infty} \frac{(\ln k)^{\de}}{k}\, \me^{-(\ln k)^{\de}/2}
    \;<\;
    \infty,
  \end{align*}
  where $k_0$ is chosen such that $k^{-\al} / 2 \leq \prob[Z_1 > k] \leq 2 k^{-\al}$ for all $k \geq k_0$.  Thus, we conclude from \cite[Theorem~3.5.1 and 3.5.2]{EKM97} that for $\prob$-a.e.\ $\om$ there exists $L_0 = L_0(\om)$ such that for all $k \geq L_0$,
  \begin{align}
    u_1(2k+1)
    \;=\;
    \big(k / (\ln k)^{\de} \big)^{1/\al}
    \;\leq\;
    \max_{-k \leq i \leq k} Z_i
    \;\leq\;
    \big(k\, (\ln k)^{1+\de} \big)^{1/\al}
    \;=\;
    u_2(2k+1).
  \end{align}
  In order to get an upper bound on the chemical distance $d^{\om}(0, L e_1)$ we consider the path $\ga_{0, L e_1} = (z_0, \ldots, z_{l})$ with length $l = 2 L_* + L$, where $L_* \ldef \argmax_{-L^{\ve} \leq i \leq L^{\ve}} Z_i$ and $\ve = 2\al / (2\al+1)$ that is chosen as follows. Starting at the origin, the path $\ga_{0, L e_1}$ goes first $|L_*|$ steps in vertical direction to $(0,L_*)$, then it goes $L$ steps in horizontal direction to $(L, L_*)$ and finally it goes again $|L_*|$ steps  in vertical direction to $L e_1$, i.e.\ for the horizontal direction the path chooses the line with an $e_2$-coordinate in $(-L^{\ve}, L^{\ve})$ which has maximal conductance.  Hence, for $L \geq L_0$,
  \begin{align*}
    d^{\om}\big(0, L e_1\big)
    &\;\leq\;
    \sum_{j=0}^{l-1} \big(1 \vee \om(z_j,z_{j+1})\big)^{-1/2}
    %\\
    \;\leq\;
    2 L^\ve + L\, \Big( \max_{-L^{\ve} \leq i \leq L^{\ve}} Z_i  \Big)^{\!-1/2}
    \\[1.5ex]
    &\;\leq\;
    2 L^\ve + L\, u_1(2L^{\ve}+1)^{-1/2}\mspace{3mu} 
    \;\leq\;
    2 L^{\ve} + L^{1 - \ve / 2\al}\, (\log L)^{\de/2\al}
  \end{align*}
  for which we conclude the claimed upper bound, since $\ve = 2\al / (2\al + 1)$.  

  In order to prove the lower bound, let $\Ga_{0, L e_1}(k)$ be the set of all paths $\ga_{0, L e_1}$ starting at the origin and ending at $L e_1$ such that $\max_{z \in \ga_{0, L e_1}} |z \cdot e_2| = k$.  Then, 
  \begin{align*}
    d^{\om}(0, Le_1)
    &\;=\;
    \inf_{k\in \bbN}\, \inf_{\ga \in \Ga_{0, L e_1}(k)}
    \sum_{(z, z') \in \ga} 1 \wedge \om(z,z')^{-1/2}
    \\[1ex]
    &\;\geq\;
    \inf_{k \in \bbN}\;
    2k + L\, \Big(\max_{-(k \vee L_0) \leq i \leq (k \vee L_0)} Z_i \Big)^{\!-1/2}
    \\[1ex]
    &\;\geq\;
    \inf_{k \in \bbN}\;
    2k + L\, \Big((k \vee L_0) \big(\ln (k \vee L_0)\big)^{1+\de}\Big)^{\!-1/2\al}.
  \end{align*}
  Set $L_* = (L / (\ln L)^{(1+\de)/(2\al)})^{2\al/(2\al+1)}$.  Notice that for $L \geq L_0$ there exist constants $0 < c < C < \infty$ such that the minimum is attained at a unique point $k_* \in [c L_*, C L_*]$.  Thus, there exists a constant $c_Z \equiv c_Z(\al,\de) > 0$ such that
  \begin{align*}
    d^{\om}(0, Le_1)
    \;\geq\;
    c_Z k_*
    \;\geq\;
    c_Z\, (\ln L)^{-(1+\de)/(2\al+1)}\, L^{2\al/(2\al+1)},
  \end{align*}
  which completes the proof of the lower bound.
\end{proof}

\section{Technical estimates}
In this section we collect some technical estimates needed in the proofs.
\begin{lemma}
  \begin{enumerate}[ (i) ]
    \item For all $a, b \geq 0$ and any $\al > 1/2$,
      \begin{align}\label{eq:rule:chain:1}
        \big(a^{\al} - b^{\al}\big)^2
        \;\leq\;
        \bigg|\frac{\al^2}{2\al-1}\bigg|\, \big(a - b\big)\,
        \big(a^{2\al-1} - b^{2\al-1}\big).
      \end{align}
    \item For $a, b \geq 0$ and any $\al \geq 1$,
      \begin{align}\label{eq:rule:chain:2}
        \big| a^{2 \al - 1} b \,-\, a b^{2 \al -1} \big|
        \;\leq\;
        \Big(1 - \frac{1}{\al} \Big)\,
        \big| a^{2 \al} - b^{2 \al} \big|
      \end{align}
    \item For all $a,b \geq 0$ and any $\al \geq 1/2$,
      \begin{align}\label{eq:rule:chain:3}
        \big(a^{2\al-1} + b^{2\al-1}\big)\, \big|a - b\big|
        \;\leq\; 
        4\,
        \big| a^{\al} -  b^{\al}\big|\,
        \big(a^{\al} + b^{\al}\big).
      \end{align}
  \end{enumerate}
\end{lemma}
\begin{proof}
  For the proof of statements  \textit{(i)} and  \textit{(iii)} we refer to \cite[Lemma A.1]{ADS13}. 

  \textit{(ii)} The statement is trivial for $a = b$ so we may assume that $0 \leq a < b$ and set $z \ldef a/b$.  Then, \eqref{eq:rule:chain:2} is equivalent to
  \begin{align*}
    \frac{z^{2\al -1}-z}{1-z^{2\al}}
    \;\leq\;
    1 - \frac{1}{\al},
    \qquad \forall\, z \in [0,1).
  \end{align*}
  But this follows from the fact that the left hand side is increasing in $z$ and converges to $1-\frac 1 \al$ as $z\to 1$.
\end{proof}

\bibliographystyle{plain}
\bibliography{literature}

\end{document}